\documentclass[twoside,openleft]{article}
\usepackage{graphicx}
\usepackage{amsmath,graphicx,amsfonts,amsthm,amssymb,setspace,mathtools,
	float,tikz,caption,fancyhdr,hyperref}
\usepackage{breakurl}
\hypersetup{pdfusetitle=true}
\PassOptionsToPackage{hyphens}{url}
\usetikzlibrary{calc}
\graphicspath{{figures/}}
\author{K\^ohei Sasaya
	\thanks{Kyoto University, Kyoto 606-8502, Japan. JSPS Research Fellow (DC1). E-mail: \texttt{ksasaya@kurims.kyoto-u.ac.jp}}
}

\title{Some inequalities between Ahlfors regular conformal dimension and spectral dimensions for resistance forms}
\date{November 21, 2022}
\makeatletter

\@addtoreset{equation}{section}
\makeatother

\let\tmp\oddsidemargin
\let\oddsidemargin\evensidemargin
\let\evensidemargin\tmp
\reversemarginpar

\newtheorem{lem}{Lemma}[section]
\newtheorem{prop}[lem]{Proposition}
\newtheorem{thm}[lem]{Theorem}
\newtheorem{cor}[lem]{Corollary}
\newtheorem*{cla}{Claim}
\theoremstyle{definition}
\newtheorem{defi}[lem]{Definition}

\newtheorem{ass}[lem]{Assumption}
\theoremstyle{remark}
\newtheorem*{rem}{Remark}

\newcommand{\qs}{\sim_{\mathrm{QS}}}

\newcommand{\ard}{\dim_\mathrm{ARC}}
\newcommand{\vd}{\mathrm{(VD)}}
\newcommand{\diam}{\mathrm{diam}}
\newcommand{\id}{\mathrm{id}}
\newcommand{\ol}[1]{\overline{#1}}
\newcommand{\ul}[1]{\underline{#1}}
\newcommand{\Mc}[1]{\mathcal{#1}}
\newcommand{\Mb}[1]{\mathbb{#1}}
\newcommand{\Mf}[1]{\mathfrak{#1}}
\newcommand{\Mr}[1]{\mathrm{#1}}
\renewcommand{\Cup}{\bigcup}
\renewcommand{\Cap}{\bigcap}
\newcommand{\scup}{\sqcup}
\newcommand{\sCup}{\bigsqcup}

\newcommand{\mim}{\mathrm{Im}}
\newcommand{\mre}{\mathrm{Re}}

\newcommand{\nai}{\mathrm{int}}

\newcommand{\doi}[1]{\href{https://doi.org/#1}{\texttt{DOI:#1}}}
\newcommand{\MR}[1]{\href{https://mathscinet.ams.org/mathscinet-getitem?mr=#1}{\texttt{MR#1}}}
\newcommand{\Zbl}[1]{\href{https://zbmath.org/?q=an:#1}{\texttt{Zbl #1}}}
\newcommand{\arX}[1]{\href{https://arxiv.org/abs/#1}{\texttt{arXiv:#1}}}
\newcommand{\same}{\ul{\hspace{2cm}}}
\begin{document}
\maketitle
\begin{abstract}
Quasisymmetric maps are well-studied homeomorphisms between metric spaces preserving annuli, and the Ahlfors regular conformal dimension $\ard(X,d)$ of a metric space $(X,d)$ is the infimum over the Hausdorff dimensions of the Ahlfors regular images of the space by quasisymmetric transformations. For a given regular Dirichlet form with the heat kernel, the spectral dimension $d_s$ is an exponent which indicates the short-time asymptotic behavior of the on-diagonal part of the heat kernel. In this paper, we consider the Dirichlet form induced by a resistance form on a set $X$ and the associated resistance metric $R$. We prove $\ard(X,R)\le \ol{d_s}<2$ for $\ol{d_s}$, a variation  of $d_s$ defined through the on-diagonal asymptotics of the heat kernel. We also give an example of a resistance form whose spectral dimension $d_s$ satisfies the opposite inequality $d_s<\ard(X,R)<2.$
\end{abstract}	
\textbf{Mathematics Subject Classification 2020:}\! Primary 30L10; Secondary 31E05, 35K08, 60J35.\\
\textbf{Key words and phrases:} Ahlfors regular conformal dimension, spectral dimension, resistance form, quasisymmetry, heat kernel.
\section{Introduction and main results}
The subject of this paper is an evaluation of a dimension of metric spaces, defined through the quasisymmetric transformations. We first recall the definition of quasisymmetry.
\begin{defi}[Quasisymmetry]
	Let $X$ be a set and $d,\rho$ be metrics on $X.$ We say $d$ is \emph{quasisymmetric} to $\rho,$ and write $d\qs\rho,$ if there exists a homeomorphism $\theta:[0,\infty)\to[0,\infty)$ such that for any $x,y,z\in X$ with $x\ne z,$
	\[\rho(x,y)/\rho(x,z)\le \theta \bigl(d(x,y)/d(x,z)\bigr).\]
\end{defi}
Roughly speaking, this definition means that an annulus in $(X,d)$ is comparable to one in $(X,\rho)$. A typical example of a metric quasisymmetric to a given metric $d$ is $d^\alpha$ for each $\alpha\in(0,1).$ It is known that $\qs$ is an equivalence relation among metrics on $X,$ and that if $d\qs\rho$ then $\rho$ induces the same topology as $d.$ \par
Quasisymmetry between general metric spaces was defined by Tukia and V\"ais\"al\"a in \cite{TV} as the analogy with the case of $\Mb{R}.$ Note that quasisymmetry on $\Mb{R}$ was discovered as a characterization of the boundary values of quasiconformal mappings from the upper half-plane to itself, by Beurling and Ahlfors in \cite{BA}, and named by Kelingos in \cite{Kel}.
Properties of quasisymmetry were well-studied in analysis on metric spaces (see \cite{Hei, Sem}, for example).  Quasisymmetry has been also used in various fields, such as heat kernel estimates (see \cite{BCM,BM,KM,Kig12,Mur}, for example) and hyperbolic group theory (see \cite{BK, BoP, MT, Pau}, for example).\par

The Ahlfors regular conformal dimension of a metric space $(X,d)$ is defined as follows. We set $B_d(x,r):=\{y\in X\mid d(x,y)<r \}$ for $x\in X$ and $r>0,$ and $\diam(X,d):=\sup_{x,y\in X}d(x,y).$
\begin{defi}[Ahlfors regular conformal dimension]\label{AR}
	For $\alpha\in(0,\infty),$ the metric $d$ is called \emph{$\alpha$-Ahlfors regular} if the $\alpha$-dimensional Hausdorff measure $\Mc{H}_\alpha$ satisfies $C^{-1}r^\alpha\le\Mc{H}_\alpha(B_d(x,r))\le Cr^\alpha$ for any $0<r\le \diam(X,d)$ and $x\in X,$ for some $C>1.$ (Note that if $(X,d)$ is $\alpha$-Ahlfors regular then $\dim_H(X,d)=\alpha,$ where $\dim_H$ is the Hausdorff dimension.) The \emph{Ahlfors regular conformal dimension} $\ard(X,d)$ of $(X,d)$ is defined by
\[
	\ard(X,d)=\inf\biggl\{\alpha\in(0,\infty) \biggm|  
	\begin{minipage}{180pt}
		there exists an $\alpha$-Ahlfors regular metric $\rho$ on $X$ with $d\qs\rho$
	\end{minipage} \biggr\}.\]

\end{defi}
This exponent implicitly appeared in Bourdon and Pajot's paper \cite{BoP} and was named by Bonk and Kleiner in \cite{BK}. In the latter paper, it was related to Cannon's conjecture, which claims that every Gromov hyperbolic group whose boundary is homeomorphic to the 2-sphere has a discrete, cocompact and isometric action on the hyperbolic 3-space $\Mb{H}^3.$ 
$\ard(X,d)$ was also characterized as a critical value related to the combinatorial $p$-modulus of a family of curves $\Gamma$ in a graph $(V, G)$ (approximating $(X,d)$), defined by
\[\mathrm{Mod}_p(\Gamma)=\inf\Bigl\{ \sum_{v\in V}f(v)^p \Bigm| f:V\to[0,\infty),\sum_{v\in\gamma}f(v)\ge1\text{ for any }\gamma\in\Gamma \Bigr\}\]
 (see \cite{Car, Kig20}). This characterization of $\ard(X,d)$ has been also used in recent studies on the construction of $p$-energies on fractals (\cite{Kig21, Shi}).
 \par
 
In \cite{Kig20}, Kigami introduced the notion of a partition satisfying the basic framework and used it to evaluate the Ahlfors regular conformal dimension of compact metric spaces. Roughly speaking, a partition satisfying the basic framework is a successive division of a given compact metric space with some good conditions. We explain this idea in the case of the Sierpi\'nski carpet. Let $Q=\{z\mid \max\{|\mre (z)|, |\mim (z)|\} \le 1/2\}\subset \Mb{C}$ and $p_j$ be the points on the boundary of $Q$ with $\arg(z)=j\pi/4$ for $1\le j\le 8$ (see Figure \ref{figQ}). We also let $\varphi_j(z)=p_j+(z-p_j)/3.$ The standard Sierpi\'nski carpet $\Mr{SC}$ is the unique nonempty compact subset of $\Mb{C}$ satisfying $ \Mr{SC}=\cup_{j=1}^8 \varphi_j( \Mr{SC})$ (see Figure \ref{FigSC}). An example of a partition $K$ of $( \Mr{SC},|\cdot|)$ satisfying the basic framework is a map from $\{\phi\}\cup\Cup_{n\ge1}\{1,...,8\}^n$ to the power set $\Mf{P}( \Mr{SC}),$ defined by 
\[K(\phi)= \Mr{SC}\text{ and }K(\{w_j\}_{j=1}^n)=\varphi_{w_1}\circ\cdots\circ\varphi_{w_n}( \Mr{SC}).\]
In \cite{Kig20}, Kigami considered the graph structure on $\{1,...,8\}^n$ for each $n$ such that there is an edge between $w,v\in\{1,...,8\}^n$ if $K(w)\cap K(v)\ne\emptyset$ and $w\ne v,$ and defined some potential theoretic exponents $\ol{d}^s_p(K),$ $\ul{d}^s_p(K)$ of this family of graphs, which he called the upper and lower $p$-spectral dimensions, for $p>0.$ See Definitions \ref{bf} and \ref{pspec} for the precise definitions of a partition satisfying the basic framework and the $p$-spectral dimensions. 
For these exponents, Kigami showed the following result.
\begin{figure}[tb]
	\centering
	\begin{minipage}{0.46\linewidth}
		\centering
		\begin{tikzpicture}[scale=0.54]
				\filldraw[lightgray] (-4,-4) -- (4,-4) -- (4,4) -- (-4,4) -- cycle ;
			\draw[ ->, >=stealth] (-5, 0) -- (5,0) node[right] {$\mre$} ;
			\draw[, ->, >=stealth] (0, -5) -- (0,5) node[above] {$\mim$} ;
			\draw[very thin] (-4,-4) -- (4,-4) -- (4,4) -- (-4,4) -- cycle ;
		
			\coordinate[label=below left:$p_1$] (p1) at (4,4);
			\coordinate[label=below left:$p_2$] (p2) at (0,4);
			\coordinate[label=below left:$p_3$] (p3) at (-4,4);
			\coordinate[label=below left:$p_4$] (p4) at (-4,0);
			\coordinate[label=below left:$p_5$] (p5) at (-4,-4);
			\coordinate[label=below left:$p_6$] (p6) at (0,-4);
			\coordinate[label=below left:$p_7$] (p7) at (4,-4);
			\coordinate[label=below left:$p_8$] (p8) at (4,0);
			\foreach \a in {1,...,8}
			\filldraw (p\a) circle [x radius=0.2, y radius=0.2];
			\coordinate[label=below right: $\frac{1}{2}$] (set) at (4,0);
			\coordinate[label=below left: $Q$] (set) at (-1.5,-1.5);
		\end{tikzpicture}
		\caption{$Q$ and $\{p_j\}_{j=1}^8.$}
		\label{figQ}
	\end{minipage}
	\begin{minipage}{0.5\linewidth}
		\centering
		\vspace{2.3\baselineskip}
		\begin{tikzpicture}[scale=0.08]
				\filldraw[black] (27,27) -- (-27,27) -- (-27,-27) -- (27,-27)-- cycle;
				\foreach \a in {1,...,81}{
					\foreach \b in {1,...,81}{
						\filldraw[fill=white, very thin] ({-247/9+2*\a/3}, {-247/9+2*\b/3}) --({-247/9+2*\a/3}, {-245/9+2*\b/3}) -- (-245/9+2*\a/3, -245/9+2*\b/3) -- (-245/9+2*\a/3, -247/9+2*\b/3) -- cycle;
				}}
				\foreach \a in {1,...,27}{
					\foreach \b in {1,...,27}{
						\filldraw[fill=white] ({-85/3+2*\a}, {-85/3+2*\b}) --({-85/3+2*\a}, {-83/3+2*\b}) -- (-83/3+2*\a, -83/3+2*\b) -- (-83/3+2*\a, -85/3+2*\b) -- cycle;
				}}
				\foreach \a in {1,...,9}{
					\foreach \b in {1,...,9}{
						\filldraw[fill=white] ({-31+6*\a}, {-31+6*\b}) --({-31+6*\a}, {-29+6*\b}) -- (-29+6*\a, -29+6*\b) -- (-29+6*\a, -31+6*\b) -- cycle;
				}}
				\foreach \a in {1,...,3}{
					\foreach \b in {1,...,3}{
						\filldraw[fill=white] ({-39+18*\a}, {-39+18*\b}) --({-39+18*\a}, {-33+18*\b}) -- (-33+18*\a, -33+18*\b) -- (-33+18*\a, -39+18*\b) -- cycle;
				}}
				\filldraw [fill=white] (-9,-9) -- (-9,9) -- (9,9) -- (9,-9) -- cycle; 
		\end{tikzpicture}
	\vspace{1.4\baselineskip}
		\caption{(Standard) Sierpi\'nski carpet.}
		\label{FigSC}
	\end{minipage}
\end{figure}
\begin{thm}[{\cite[Theorem 4.7.9]{Kig20} and \cite[Theorem 3.9]{Sas1}}]\label{Kmain}
	Let $(X,d)$ be a metric space with a partition $K$ satisfying the basic framework. Then for $p>0,$
	\begin{enumerate}
		\item if $p>\ard(X,d)$ then $p>\ol{d}^s_p(K)\ge\ul{d}^s_p(K)\ge \ard(X,d).$
		\item If $p\le\ard(X,d)$ then $p\le\ul{d}^s_p(K)\le\ol{d}^s_p(K)\le\ard(X,d).$
	\end{enumerate}
\end{thm}
Note that the assumption in Theorem \ref{Kmain} is slightly different from that in \cite[Theorem 4.7.9]{Kig20}, but it is justified by \cite[Theorem 4.7.6]{Kig20}. We also note that the contribution of \cite{Sas1} was an extension of the framework and the result to non-compact spaces. We emphasize that the $p$-spectral dimensions depend only on the given metric space and the partition, and do not have any stochastic characterization. However, it was pointed out in \cite{Kig20} that if $(X,d)$ is the Sierpi\'nski gasket or a generalized Sierpi\'nski carpet with the Euclidean metric and $K$ is the canonical partition as described above, then $\ol{d}^s_2(K)$ and $\ul{d}^s_2(K)$ coincide with the spectral dimension, defined as follows, of the standard Dirichlet form.
\begin{defi}[Spectral dimension]
Let $(X,d)$ be a locally compact separable metric space, $\mu$ be a Radon measure on $X$ with full support, and $(\Mc{E,F})$ be a regular Dirichlet form on $L^2(X,\mu).$ We assume that $(\Mc{E,F})$ has the associated \emph{heat kernel} (or \emph{transition density}), namely, a (jointly) continuous function $p(t,x,y):(0,\infty)\times X\times X\to[0,\infty)$ such that
\[T_tu(x)=\int_X p(t,x,y)u(y)d\mu(y)\quad\text{for $\mu$-a.e. $x\in X$}\]
for any $t\in(0,\infty)$ and any $u\in L^2(X,\mu),$ where $\{T_t\}_{t\in(0,\infty)}$ denotes the Markovian semigroup on $L^2(X,\mu)$ associated with $(\Mc{E,F}).$ The limit
\[d_s(\mu,\Mc{E,F})=d_s(X,\mu,\Mc{E,F}):=-2\lim_{t\downarrow0}\frac{\log p(t,x,x)}{\log t}\]
is called the \emph{spectral dimension} of the regular Dirichlet space $(X,\mu,\Mc{E,F}),$ if it exists and is independent of a choice of $x\in X.$
\end{defi}
In this paper, we prove an inequality similar to the case of $p=2$ of Theorem \ref{Kmain} (1), between the Ahlfors regular conformal dimension and a variation of the spectral dimension defined through the on-diagonal asymptotics of the heat kernel, for the case where the Dirichlet form is induced by a resistance form, defined as follows.

\begin{defi}[Resistance form]
	Let $X$ be a set, $\Mc{F}$ be a linear subspace of the space $\ell(X)$ of $\Mb{R}$-valued functions on $X,$ and $\Mc{E}$ be a nonnegative quadratic form on $\Mc{F}.$ The pair $(\Mc{E,F})$ is called a \emph{resistance form} on $X$ if it satisfies the following conditions. 
	\begin{align}
		\bullet\ & 1_X\in\Mc{F},\text{ and }\Mc{E}(u,u)=0\text{ if and only if $u$ is constant.}\label{RF1} \\
		\bullet\ &\text{Define an equivalence relation $\sim$ as $u\sim v$ if and only if $u-v$ is}\notag\\
		& \text{constant, then $(\Mc{F}/\!\sim,\Mc{E})$ is a Hilbert space.} \label{RF2} \\
		\bullet\ & \text{If $x\ne y$ then there exists $u\in\Mc{F}$ with $u(x)\ne u(y).$} \\
		\bullet\ &R(x,y)\!:=\bigl(\inf\{\Mc{E}(u,u)\mid u\in\Mc{F},\ u(x)=1,\ u(y)=0\}\bigr)^{-1}<\infty \label{RF3}\text{ if }x\ne y. \\
		\bullet\ &\text{If $u\in\Mc{F}$ then $\hat u:=\max\{0,\min\{1,u\}\}\in\Mc{F}$ and $\Mc{E}(\hat u,\hat u)\le \Mc{E}(u,u).$} \label{RF4}
	\end{align}
We define $R(x,x)=0$ for $x\in X.$
\end{defi}
One of the most basic properties of a resistance form is that the infimum in \eqref{RF3} is attained and defines a metric $R$ on $X$, called the \emph{resistance metric} associated with the resistance form. The notion of resistance form was introduced in \cite{Kig95}. Typical examples of the Dirichlet forms induced by resistance forms are the standard Dirichlet forms on ``low-dimensional'' fractals, such as p.c.f. self-similar fractals. This framework includes Dirichlet forms whose associated Hunt processes have jumps (see \cite[Chapter 16]{Kig12}, for example). Moreover, there are also examples of resistance forms on some spatially inhomogeneous fractals (see Figure \ref{nonsym} and \cite{Ham}, for example) and more general sets (see \cite{Cro}, for example).\par
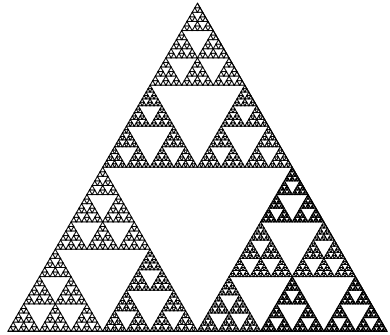
\begin{figure}[tb]
	\centering
	\begin{tikzpicture}[scale=0.14]
				\foreach \a/\b/\c in {0/0/9,0/9/9,0/27/9,18/0/6,18/6/6,24/6/6}{
						\coordinate (p\a\b) at ({\a+\b*cos(60)+\c/2},{\b*sin(60)+\c*sqrt(3)/6});
						\foreach \d/\e in {30/1, 90/2, 150/1, 210/2, 270/1, 330/2}{
								\foreach \f/\g in {30/1, 90/2, 150/1, 210/2, 270/1, 330/2}{
										\foreach \h in {90, 210, 330}{
												\draw[very thin] ($(p\a\b)+\c*\e*({sqrt(3)*cos(\d)/9},{sqrt(3)*sin(\d)/9})+\c*\g*({sqrt(3)*cos(\f)/27},{sqrt(3)*sin(\f)/27})+\c*({sqrt(3)*cos(\h)/54}, {sqrt(3)*sin(\h)/54})-\c*({sqrt(3)*cos(30)/54},{sqrt(3)*sin(30)/54})$) -- ++ ({\c/18},0) --++ ({cos(120)*\c/18},{sin(120)*\c/18})--cycle ;
								}}}}
				\foreach \a/\b/\c in {9/0/9,24/0/6,30/0/6,18/12/6,0/18/9,9/18/9}{
						\coordinate (p\a\b) at ({\a+\b*cos(60)+\c/2},{\b*sin(60)+\c*sqrt(3)/6});
						\foreach \d in {90, 210, 330}{
								\foreach \e in {90, 210, 330}{
										\foreach \f/\g in {30/1, 90/2, 150/1, 210/2, 270/1, 330/2}{
												\foreach \h in {90, 210, 330}{
														\draw[very thin] ($(p\a\b)+\c*({sqrt(3)*cos(\d)/6},{sqrt(3)*sin(\d)/6})+\c*({sqrt(3)*cos(\e)/12},{sqrt(3)*sin(\e)/12})+\c*\g*({sqrt(3)*cos(\f)/36},{sqrt(3)*sin(\f)/36})+\c*({sqrt(3)*cos(\h)/72},{sqrt(3)*sin(\h)/72})-\c*({sqrt(3)*cos(30)/72},{sqrt(3)*sin(30)/72})$) -- ++ ({\c/24},0) --++ ({cos(120)*\c/24},{sin(120)*\c/24})--cycle ;
									}}}}}
		\draw(0,0)--(36,0) -- (18, {sqrt(3)*18});
	\end{tikzpicture}
	\caption{Inhomogeneous fractal, having a canonical resistance form.}
	\label{nonsym}
\end{figure}
In the remainder of this section except for Subsection \ref{secpre}, we make the following assumption, which is needed for our main theorem.
\begin{ass}\label{ass} 
	$(\Mc{E,F})$ is a resistance form on a set $X,$ and the resistance metric $R$ associated with  $(\Mc{E,F})$ is complete and satisfies $\ard(X,R)<\infty.$ 
\end{ass}
We note that the condition $\ard(Y,\rho)<\infty$ for a metric space $(Y,\rho)$ has simple geometric characterizations which may be easily checked (see Theorem \ref{xbf}). In particular, under Assumption \ref{ass}, there exists a partition $K$ of $(X,R)$ satisfying the basic framework. Let us recall the definition of the volume doubling property.
\begin{defi}[Volume doubling property]
	A Borel measure $\mu$ on a metric space $(Y,\rho)$ has the \emph{volume doubling property} with respect to $\rho$ if
	\[0<\mu(B_\rho(x,2r))\le C\mu(B_\rho(x,r))<\infty\]
	for any $x\in Y$ and $r>0,$ for some $C>1.$ Then we say $\mu$ is (VD)$_\rho$ for short. We write $\Mc{M}_{(Y,\rho)}$ for the set of all Borel measures on $(Y,\rho)$ that are (VD)$_\rho.$ 
\end{defi}
For any $\mu\in\Mc{M}_{(X,R)},$ we can check that the assumptions of \cite[Theorems 9.4 and 10.4]{Kig12} are satisfied (see Proposition \ref{pre}) and obtain the following lemma.
\begin{lem}\label{xhk}
	Let $\mu\in\Mc{M}_{(X,R)}.$ For $u,v\in \Mc{F}\cap L^2(X,\mu),$ we define $\Mc{E}_{\mu,1}(u,v)$ by
	\[\Mc{E}_{\mu,1}(u,v)=\Mc{E}(u,v)+\int_X uvd\mu,\]
	then $(\Mc{F}\cap L^2(X,\mu), \Mc{E}_{\mu,1})$ is a Hilbert space.
	Let $\Mc{D}_\mu$ be the closure of $\Mc{F}\cap C_0(X)$ with respect to $\Mc{E}_{\mu,1},$ and $\Mc{E}_\mu=\Mc{E}|_{\Mc{D}_\mu\times \Mc{D}_\mu},$ where $C_0(X)$ is the set of all continuous functions on $(X,R)$ whose supports are compact. Then $(\Mc{E}_\mu, \Mc{D}_\mu)$ is a regular Dirichlet form on $L^2(X,\mu).$ Moreover, the associated heat kernel $p_\mu(t,x,y)$ exists.
\end{lem}
The main theorem of this paper is the following.
\begin{thm}\label{main}
	Let $\mu\in\Mc{M}_{(X,R)}.$ Then the limit 
	\begin{equation}\label{eqmain}
		\ol{d_s}(\mu,\Mc{E}_\mu,\Mc{D}_\mu):=\lim_{t\to\infty} \sup_{x\in X, s\in(0,\diam(X,R))}2\frac{\log\bigl(p_\mu(s/t,x,x)/p_\mu(s,x,x)\bigr)}{\log t}
	\end{equation}
exists and satisfies $\ard(X,R)\le\ol{d_s}(\mu,\Mc{E}_\mu,\Mc{D}_\mu)<2.$
\end{thm}
Note that if $\diam(X,R)<\infty$ then $\diam(X,R)$ in the right-hand side of \eqref{eqmain} can be replaced by $1$ because of Lemma \ref{exp} (5) and (6).\par
The following theorem is needed to prove Theorem \ref{main}, and it characterizes $\ol{d}^s_2$ if $(\Mc{E,F})$ is local. Recall from \cite[Definition 7.5]{Kig12} that $(\Mc{E,F})$ is said to be \emph{local} if $\Mc{E}(u,v)=0$ whenever $u,v\in\Mc{F}$ and $\inf_{x,y\in X,u(x)v(y)\ne 0}R(x,y)>0$ (see also Definition \ref{local}). We also recall that $(X,R)$ has a partition satisfying the basic framework by Theorem \ref{xbf}.
\begin{thm}\label{main2}
	Let $K$ be a partition of $(X,R)$ satisfying the basic framework. Then $\ol{d}^s_2(K)\le \ol{d_s}(\mu,\Mc{E}_\mu,\Mc{D}_\mu)$ for any $\mu\in\Mc{M}_{(X,R)}.$ Moreover,  if $(\Mc{E,F})$ is local then $\inf_{\mu\in\Mc{M}_{(X,R)}}\ol{d_s}(\mu,\Mc{E}_\mu,\Mc{D}_\mu)=\ol{d}^s_2(K).$
\end{thm}
If $(\Mc{E,F})$ and $\mu$ are the standard resistance form and the standard measure on the Sierpi\'nski gasket or a generalized Sierpi\'nski carpet with $\ard(X,d)<2$ where $d$ is the Euclidean metric on $X,$ then $\ol{d_s}(\mu,\Mc{E}_\mu,\Mc{D}_\mu)$ coincides with $d_s(\mu,\Mc{E}_\mu,\Mc{D}_\mu).$ Therefore in these cases Theorem \ref{main} yields $\ard(X,d)\le d_s(\mu,\Mc{E}_\mu,\Mc{D}_\mu)<2$ because $d\qs R,$ which recovers the result obtained in \cite{Kig20} as an application of Theorem \ref{Kmain} (1).\par
In general, $\ol{d_s}(\mu,\Mc{E}_\mu,\Mc{D}_\mu)$ does not coincide with $d_s(\mu,\Mc{E}_\mu,\Mc{D}_\mu)$ even if the latter exists. Moreover, the analog of the inequality in Theorem \ref{main} is false in general for the latter, as the following theorem states.
\begin{thm}\label{cex}
	There exist $X,$ $(\Mc{E,F})$ (satisfying Assumption \ref{ass}) and $\mu\in\Mc{M}_{(X,R)},$ such that $d_s(\mu,\Mc{E}_\mu,\Mc{D}_\mu)$ exists and \[d_s(\mu,\Mc{E}_\mu,\Mc{D}_\mu)<\ard(X,R)<2.\]
\end{thm}
We briefly describe the set $X$ on which we will construct the example of Theorem \ref{cex}. Following a particular rule, we use either the cell subdivision rule of $\Mr{SC}$ or that of the Vicsek set (that is, the unique nonempty compact subset $\Mr{VS}$ of $\Mb{C}$ with $\Mr{VS}=\cup_{j=1,3,5,7}\varphi_j(\Mr{VS})\Cup\frac{1}{3}\Mr{VS}$) for each scale, and obtain $X$ as the resulting set (see Section \ref{seccex} for details). $X$ has the full symmetry of the unit square, but is not exactly self-similar (see Figure \ref{FigCex}). In this example, we also show that the resistance metric $R$ associated with $(\Mc{E,F})$ is quasisymmetric to the Euclidean metric on $X$ (see Theorem \ref{cres}).\par
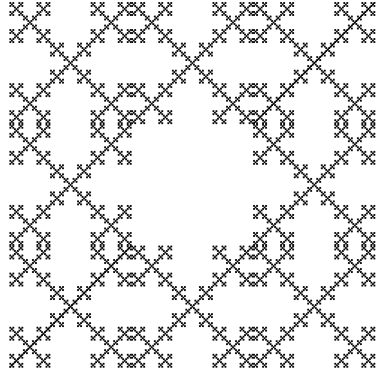
\begin{figure}[tb]
	\centering
	\begin{tikzpicture}[scale=0.03]
				\foreach \a/\A in {54/0, 54/54,0/54, -54/54,-54/0, -54/-54, 0/-54,54/-54}{
						\foreach \b / \B in {0/0, 18/18, -18/18, -18/-18, 18/-18}{		
								\foreach \c / \C in {0/0, 6/6, -6/6, -6/-6, 6/-6}{	
										\foreach \d /\D in {0/0, 2/2, -2/2, -2/-2, 2/-2}{
												\foreach \e/\E in {0/0, {2/3}/{2/3},{-2/3}/{2/3},{-2/3}/{-2/3},{2/3}/{-2/3}}{
													\draw[very thin] (\a+\b+\c+\d+\e+1/3,\A+\B+\C+\D+\E+1/3) -- (\a+\b+\c+\d+\e-1/3,\A+\B+\C+\D+\E-1/3) ;
													\draw[very thin] (\a+\b+\c+\d+\e+1/3,\A+\B+\C+\D+\E-1/3) -- (\a+\b+\c+\d+\e-1/3,\A+\B+\C+\D+\E+1/3) ;
								}}}}}
		\draw[very thin] (-81,-81)--(-27,-27);
		\draw[very thin] (27,27)--(81,81);
	\end{tikzpicture}
	\caption{Example of Theorem \ref{cex}.}
	\label{FigCex}
\end{figure}
The structure of this paper is as follows. Section \ref{secres} is devoted to proving inequalities of resistances used in the later sections. In Section \ref{secpart} we introduce the precise definition of a partition satisfying the basic framework and show related inequalities. We prove Theorems \ref{main} and \ref{main2} in Section \ref{secpf}, and Theorem \ref{cex} in Section \ref{seccex}. Appendix \ref{ApLoc} is devoted to proving the equivalence between different formulations of the  local property of a resistance form.

\subsection{Notation}\label{secpre}
Throughout this paper, we use the following notation.
\begin{itemize}
	\item The letter $\#$ denotes the cardinality of sets, and $\Mf{P}$ denotes the power set of sets. 
	\item For a set $X,$ we denote by $\ell(X)$ the set of all $\Mb{R}$-valued functions on $X.$
	 \item $a\vee b=\max\{a,b\}$ and $a\wedge b=\min\{a,b\}$ for $a,b\in\Mb{R}$ (or $\Mb{R}$-valued functions).
	 \item By abuse of notation, we write $x$ instead of $\{x\}$ if no confusion can arise. For example, we write $f^{-1}(x)$ instead of $f^{-1}(\{x\}).$
	 \item For a set $X$ and $A\subset X,$ we write $A^c$ instead of $X\setminus A$ if the whole set $X$ is obvious.
	 \item Let $(X,d)$ be a metric space. For $A\subset X,$ we will denote by $\nai(A)$ the interior of $A$ and by $\ol{A}$ the closure of $A$. Moreover, for any Borel measure $\mu$ on $X,$ we write $V_{d,\mu}(x,r)=\mu(B_d(x,r))$ (where $B_d(x,r)=\{y\in X\!\mid\! d(x,y)<r\}$). We will omit subscripts of $V_{d,\mu}(x,r)$ and $B_d(x,r)$ if the metric and/or measure is obvious.
	 \item Let $X$ be a set and $f:X\to X$ be a map, then we set $f^k:=\overbrace{f\circ\cdots\circ f}^{k}$ for $k>0$ and $f^0:=\Mr{id}_X.$ Moreover, $f^{-k}$ denotes $(f^k)^{-1}$ for $k>0.$
	\item Let $f,g$ be functions on a set $X$ and $A\subset X.$ We say $f(x)\lesssim g(x)$ (resp. $f(x)\gtrsim g(x)$) for any $x\in A$ if there exists $C>0$ such that $f(x)\le Cg(x)$ (resp. $f(x)\ge Cg(x)$) for any $x\in A.$ We also write $f(x)\asymp g(x)$ (for any $x$) if $f(x)\lesssim g(x)$ and $f(x)\gtrsim g(x).$ Note that we will not use this notation when we want to stress the constant $C.$ 
	\item Let $f$ be a function or variable defined by some type of maximum or minimum over a set of functions. Then we say $g$ is the optimal function for $f$ if $g$ attains the maximum or minimum. For example, let $R$ be the resistance metric associated with $(\Mc{E,F}),$ then the optimal function $u$ for $R(x,y)$ is such that $u\in \Mc{F},$ $u(x)=1,$ $u(y)=0$ and $\Mc{E}(u,u)=R(x,y)^{-1}.$
\end{itemize}
\section{Resistance forms}\label{secres}
In this section, we prove some properties of resistance forms and associated heat kernels, which we will use in the proof of Theorem \ref{main} and related statements. We first note the difference between two types of resistances between subsets. Throughout the rest of this paper, $(\Mc{E,F})$ denotes a resistance form on a set $X$ and $R$ denotes the associated resistance metric.
\begin{lem} Let $A,B\subset X$ be nonempty. If 
	\[\Mc{F}_{A,B}:=\{u\in\Mc{F}\mid u|_A\equiv 1,\ u|_B\equiv 0\}\ne\emptyset,\]
	then $\min_{u\in F_{A,B}}\Mc{E}(u,u)$ exists and $u\in F_{A,B}$ attaining the minimum is unique.
\end{lem}
\begin{proof}
	Fix any $x\in B,$ then we have
	\[|u(y)-v(y)|=|(u-v)(x)-(u-v)(y)|\le \Mc{E}(u-v,u-v)^{1/2}R(x,y)^{1/2}\]  
	for any $y\in X$ and $u,v\in\Mc{F}$ with $u(x)=v(x)=0.$ This shows that $\Mc{F}_{A,B}$ is a closed convex subset of the Hilbert space $(\{u\in\Mc{F}\mid u(x)=0\},\Mc{E})$ and the claim follows.
\end{proof}
\begin{defi}[Resistance between sets]
Let $\Mc{R}(A,B)$ denote the reciprocal of $\min_{u\in F_{A,B}}\Mc{E}(u,u)$ if $\Mc{F}_{A,B}\ne\emptyset$ and $0$ if $\Mc{F}_{A,B}=\emptyset.$ We also define $\Mc{R}(A,B)=\infty$ if $A=\emptyset$ or $B=\emptyset$ for ease of notation. We call $\Mc{R}(A,B)$ the \emph{resistance between sets} $A$ and $B$, associated with $(\Mc{E,F}).$
\end{defi}
On the other hand, we use the notation $R(A,B)$ for the resistance metric between sets, that is, $R(A,B)=\inf_{x\in A, y\in B}R(x,y).$ Note that $R(x,y)=\Mc{R}(\{x\},\{y\})$ for any $x,y\in X$ and $\Mc{R}(A,B)\le R(A,B)$ for any $A,B\subset X$ but $R(A,B)\ne\Mc{R}(A,B)$ in general. Throughout this paper $\Mc{R}$ denotes the resistance between sets associated with $(\Mc{E,F}),$ and more generally, the letter $R$ is used for resistance metrics (``$R_n$'' for example) and $\Mc{R}$ is used for resistances between sets (``$\Mc{R}_n$'' for example).\par
Our next aim is to prove Lemma \ref{xhk}. For this purpose, we first introduce some notions of a metric space and a resistance form. 
\begin{defi}[Doubling, uniformly perfect] Let ($Y,\rho$) be a metric space.
	\begin{enumerate}
		\item $(Y,\rho)$ is called \textit{doubling} if there exists $N\in\Mb{N}$ such that for any $x\in Y$ and $r>0,$ there exist $\{x_i\}_{i=1}^N\subset Y$ with $B(x,2r) \subset \Cup_{i=0}^N B(x_i,r).$
		\item $(Y,\rho)$ is called \emph{uniformly perfect} if there exists $\gamma>1$ such that $B(x,\gamma r)\setminus B(x,r)\ne\emptyset$ whenever $B(x,r)\ne Y.$
	\end{enumerate}
\end{defi}
\begin{rem}
	It is easy to see that a doubling metric space is separable.
\end{rem}
\begin{defi}[regular]
	$(\Mc{E,F})$ is called \emph{regular} if $\Mc{F}\cap C_0(X,R)$ is dense in $C_0(X,R)$ with respect to the supremum norm. 
\end{defi}
\begin{defi}[Annulus comparable condition]
	We say that $(\Mc{E,F})$ satisfies the \emph{annulus comparable condition}, (ACC) for short, if there exists $\alpha>1$ such that
	$\Mc{R}(x,B(x,r)^c)\gtrsim \Mc{R}(x,B(x,r)^c\cap\ol{B(x,\alpha r)})$
	for any $x\in X$ and $r>0.$
\end{defi}
Note that the inverse direction of the above inequality immediately follows from the inclusion of sets. It is easy to see that if (ACC) holds then $(X,R)$ is uniformly perfect.\par
Hereafter, we make Assumption \ref{ass} to the end of Section \ref{secpf}. Note that by Theorem \ref{xbf}, both $R$ and $d$ are doubling and uniformly perfect.
\begin{prop}\label{pre}
$(\Mc{E,F})$ is regular and satisfies (ACC).
\end{prop}
For the proof of this proposition, we introduce some results.
\begin{prop}[{\cite[Theorem 8.4]{Kig12}}]\label{Proptr}
	Let $Y$ be a nonempty subset of X. Define $	\Mc{F}|_Y=\{u|_Y\mid u\in\Mc{F} \}$ and  \begin{equation}
	\Mc{E}|_Y(u_*,u_*)=\inf\{\Mc{E}(u,u)\mid u\in\Mc{F},\ u_*=u|_Y\} \label{ey}
	\end{equation}
for any $u_*\in\Mc{F}|_Y,$  then the infimum of \eqref{ey} is attained. Moreover, there exists a unique extension of $\Mc{E}|_Y$ to $\Mc{F}|_Y\times \Mc{F}|_Y$ such that $(\Mc{E}|_Y, \Mc{F}|_Y)$ is a resistance form. In particular,  if $\#(Y)<\infty$ then $\Mc{F}|_Y=\ell(Y).$
\end{prop}
\begin{defi}
$(\Mc{E}|_Y, \Mc{F}|_Y)$ is called the \textit{trace} of $\Mc{(E,F)}$ on $Y.$
\end{defi}
\begin{rem}
	In \cite[Theorem 8.4]{Kig12}, $R$ is assumed to be separable and complete, and it is so in our case. However, by the standard argument in Hilbert space theory, it is easy to show that Proposition \ref{Proptr} is also true for resistance forms whose associated resistance metric is not necessarily separable and complete (see \cite[Theorem 2.29]{Kaj}, for example).
\end{rem}
\begin{defi}
		We define the following terminologies for abbreviation. 
		\begin{enumerate}
			\item We say $\{V_n\}_{n\ge0}$ is a \emph{spread sequence} of a metric space $(Y,\rho)$ if it is an increasing sequence of nonempty finite subsets satisfying $\ol{\cup_{n\ge0}V_n}=Y.$
			\item Assume that $V$ is a finite set and  $(\Mc{E},\ell(V))$ is a resistance form on $V.$ We call $\mu=\{\mu_{x,y}\}_{x,y\in V} \subset \Mb{R}^{V\times V}$ the \emph{resistance weights} associated with $(\Mc{E},\ell(V))$ if 
			$\mu_{x,x}=-\sum_{z:z\ne x}\mu_{x,z}$ and $\mu_{x,y}=\mu_{y,x}$ for any $x,y\in V,$ and \[\Mc{E}(u,v)=\frac{1}{2}\sum_{x,y\in V}(u(x)-u(y))(v(x)-v(y))\mu_{x,y} \text{ for any } u,v\in\ell(V).\]
		\end{enumerate}
\end{defi}
Moreover, we write $\Mc{R}_Y(A,B)$ (resp. $R_Y(A,B)$) instead of $\Mc{R}_Y(A\cap Y,B\cap Y)$ (resp. $R_Y(A\cap Y,B\cap Y)$) for abbreviation, where $\Mc{R}_Y$ (resp. $R_Y$) is the resistance between sets (resp. resistance metric) associated with the trace of $(\Mc{E,F})$ on $Y.$
\begin{rem}
	By \cite[Proposition 2.1.3]{Kig01}, for any  $(\Mc{E},\ell(V)),$ the unique resistance weights associated with that exist. Moreover, $\mu_{x,y}\ge0$ for any $x,y\in V$ with $x \ne y$ and $\mu_{x,x}<0$ for any $x\in V$ by the same proposition and \eqref{RF1}.
\end{rem}
\begin{prop}[{\cite[Section 2.3]{Kig01}}]\label{comp}
	Assume that $\{V_n\}_{n\ge0}$ is a spread sequence, then
	\[\Mc{F}=\{u\mid u\in C(X,R),\ \lim_{n\to\infty}\Mc{E}|_{V_n}(u|_{V_n},u|_{V_n})<\infty\}.\]
	Moreover, $\Mc{E}(u,v)=\lim_{n\to\infty}\Mc{E}|_{V_n}(u|_{V_n},v|_{V_n})$ for any $u,v\in\Mc{F}.$
\end{prop}
\begin{rem}
	$\{\Mc{E}|_{V_n}(u|_{V_n},u|_{V_n})\}_{n\ge0}$ is an increasing sequence for any $u\in \Mc{F}$ by definition.
\end{rem}
\begin{lem}\label{inf}
	Let $\{f_n\}_{n\ge0}\subset\Mc{F}$ with $\sum_{n\ge0}\Mc{E}(f_n,f_n)<\infty.$
	\begin{enumerate}
		\item If $\sup_{n\ge0}f_{n}(x_*)<\infty$ for some $x_*\in X,$ then $\bar f:=\sup_{n\ge0}f_n\in\Mc{F}$ and $\Mc{E}(\bar f,\bar f)\le \sum_{n\ge0}\Mc{E}(f_n,f_n).$
		\item If $\inf_{n\ge0}f_{n}(x_*)>-\infty$ for some $x_*\in X,$ then $\ul f:=\inf_{n\ge0}f_n\in\Mc{F}$ and $\Mc{E}(\ul f,\ul f)\le \sum_{n\ge0}\Mc{E}(f_n,f_n).$
	\end{enumerate}
\end{lem}
\begin{rem}
	Lemma \ref{inf} is essentially a special case of \cite[Theorem 2.38 (1)]{Kaj}, which is written in Japanese. We give a proof of Lemma \ref{inf} for the reader's convenience.
\end{rem}
\begin{proof}
	Replacing $f_n$ by $-f_n,$ we only need to show $(1).$ We first note that
	\begin{align*}
	 |\bar f(x)-\bar f(y)|\le \sup_{n}|f_n(x)-f_n(y)|&\le \sup_n(R(x,y)\Mc{E}(f_n,f_n))^{1/2}\\
	 &\le R(x,y)^{1/2}(\textstyle \sum_{n\ge0}\Mc{E}(f_n,f_n))^{1/2}
	\end{align*}
	for any $x,y\in X,$ so $\bar f(x)<\infty$ and $\bar f\in C(X,R).$ Let $\{V_m\}_{m\ge1}$ be a spread sequence of $(X,R)$  and $\mu_m$ be associated resistance weights with $(\Mc{E}|_{V_m},\ell(V_m)).$ Then
	\begin{align*}
		\Mc{E}(\bar f, \bar f)=&\lim_{m\to\infty}\Mc{E}|_{V_m}(\bar f|_{V_m},\bar f|_{V_m})\\
		=&\lim_{m\to\infty}\frac{1}{2}\sum_{x,y\in V_m: x\ne y}\bigl(\bar f(x)-\bar f(y)\bigr)^2(\mu_m)_{x,y}\\
		\le &\lim_{m\to\infty}\frac{1}{2} \sum_{n\ge0}\sum_{x,y\in V_m: x\ne y}\bigl(f_n(x)-f_n(y)\bigr)^2(\mu_m)_{x,y}\\
		=& \lim_{m\to\infty}\sum_{n\ge0}\Mc{E}|_{V_m}(f_n|_{V_m},f_n|_{V_m}) =\sum_{n\ge0}\Mc{E}(f_n,f_n).
	\end{align*}
	(Note that any term of the sums in the above inequalities is nonnegative.) This with Proposition \ref{comp} proves the lemma. 
\end{proof}
\begin{proof}[Proof of Proposition \ref{pre}]
	By \cite[Lemma 7.10]{Kig12}, there exist $n\ge0$ and $C>0$ with
	\[C^{-1}2^k\le \Mc{R}(x, B(x,2^k)^c\cap\ol{B(x,2^{k+N})})\le C2^k\]
	for any $x\in X$ and $k\in\Mb{Z}$ unless $B(x,2^k)=X.$ Fix any $x\in X$ and set $A_k=B(x,2^k)^c\cap\ol{B(x,2^{k+N})}.$ Let $f_k$ be the optimal functions for $\Mc{R}(x,A_k)$ if $B(x,2^k)\ne X,$ and otherwise $f_k\equiv 1.$ 
	Then for any $a\in\Mb{Z},\ \sum_{k\ge a}\Mc{E}(f_k,f_k)<C2^{-a+1}$ and so $g_a:=\inf_{k \ge a}f_k\in\Mc{F}$ by Lemma \ref{inf} (2). Since $g_a(x)=1$ and $g_a|_{B(x,2^a)^c} \equiv 0,$ it follows that 
\begin{equation}\label{eq2}
	\Mc{R}(x, B(x,2^a)^c)>C^{-1}2^{a-1} \text{ for any }a\in\Mb{Z}\text{ with }B(x, 2^a)\ne X,
\end{equation}
which shows (ACC). Moreover, \eqref{eq2} also shows that for any nonempty $Y\subset X$ with $x\not\in \ol{Y},$ there exists $f\in \Mc{F}$ satisfying $f(x)=1$ and $f|_Y\equiv 0.$ Applying \cite[Theorem 6.3]{Kig12}, we conclude that $(\Mc{E,F})$ is regular. 
\end{proof}
Now Lemma \ref{xhk} immediately follows from \cite[Theorems 9.4 and 10.4]{Kig12} with Proposition \ref{pre}. (Note that the condition (ACC) is used later.)\par
We next give some properties of resistance forms, which will be needed in Section \ref{secpart}.
\begin{lem}\label{resab}
	Let $A_1, A_2$ be nonempty subsets of $X$ and $\{V_n\}_{n\ge0}$ be a spread sequence. Suppose that $A_i\subset\ol{\cup_{n\ge0}(A_i\cap V_n)}$ for $i=1,2.$ Then $\Mc{R}(A_1,A_2)$ $=\lim_{n\to\infty}\Mc{R}_n(A_1,A_2),$ where $\Mc{R}_n$ is the resistance between sets associated with $(\Mc{E}|_{V_n},\Mc{F}|_{V_n}).$
\end{lem}
\begin{proof}
	By definition of $\Mc{E}|_{Y_n},$ it suffices to show that \[\textstyle \Mc{R}(A_1,A_2)\ge\lim_{n\to\infty}\Mc{R}_n(A_1,A_2).\]
	We may assume $\lim_{n\to\infty}\Mc{R}_n(A_1,A_2)>0$ and $A_i\cap V_0\ne\emptyset\ $ for $i=1,2$ without loss of generality. Let $\{f_n\}_{n\ge0}\subset\Mc{F}$ be functions satisfying
	\begin{align*}
		&\Mc{R}_n(A_1,A_2)^{-1}\\
		=&\!\min\Bigl\{ \min\bigl\{ \Mc{E}(f,f)\bigm\vert f\in\Mc{F}, f|_{V_n}\equiv f_*  \bigr\}\!\Bigm\vert\! f_*\in\ell(V_n), f_*|_{A_1\cap V_n}\equiv1, f_*|_{A_2\cap V_n}\equiv 0\Bigr\}\\
	=&\Mc{E}(f_n,f_n),
	\end{align*} 
$f_n|_{A_1\cap V_n}\equiv 1$ and $f_n|_{A_2\cap V_n}\equiv 0.$ Then by the convexity argument, we obtain $0\le \Mc{E}(f_n-f_m, f_n-f_m)=\Mc{E}(f_n,f_n)-\Mc{E}(f_m,f_m)$ for any $n,m\in\Mb{N}$ with $n>m.$ Fix any $x\in A_1\cap V_1.$ Since $\lim_{n\to\infty}\Mc{E}(f_n,f_n)=\lim_{n\to\infty}\Mc{R}_n(A_1,A_2)^{-1}<\infty,$ there exists $f\in\Mc{F}$ such that $f(x)=1$ and $\lim_{n\to\infty}\Mc{E}(f-f_n,f-f_n)=0$ by \eqref{RF2}. Then for any $y\in \cup_{n\ge0}(A_1\cap V_n),$\[
	|f(y)-1|=|(f-f_n)(y)-(f-f_n)(x)|\le R(x,y)\Mc{E}(f-f_n,f-f_n).\]
Hence $f|_{A_1}\equiv 1$ because the right hand side of the last inequality tends to 0 and $f\in C(X,R).$ In the same way, we obtain $f|_{A_2}\equiv 0$ and so 
\[\Mc{R}(A_1,A_2)^{-1}\le\Mc{E}(f,f)=\lim_{n\to\infty}\Mc{E}(f_n,f_n)=\lim_{n\to\infty}\Mc{R}_n(A_1,A_2)^{-1},\]
which completes the proof.
\end{proof}
\begin{lem}\label{resbb}
	Let $\eta\in(0,1),$ then $\Mc{R}(\ol{B(x,\eta r)}, B(x,r)^c)\asymp r$ for any $x\in X, r>0$ with $B(x,r)\ne X.$
\end{lem}
\begin{proof}
	By Proposition \ref{pre} and \cite[Theorem 7.12]{Kig12}, there exists $C>0$ such that
	$C^{-1}r \le \Mc{R}(x,B(x,r)^c)\le Cr$ for any $x\in X, r>0$ with $B(x,r)\ne X.$ Hence we only need to show $\Mc{R}(\ol{B(x,\eta r)}, B(x,r)^c)\gtrsim r.$ \\
	We first prove for the case $\eta\le 1/2C.$ Let $f_{x,r}$ be the optimal function of $\Mc{R}(x,B(x,r)^c),$ then by \cite[Theorems 4.1 and 4.3 and Lemma 4.5]{Kig12},
	\begin{align*}
		f_{x,r}(y)\ge&\frac{\Mc{R}(x,B(x,r)^c)+\Mc{R}(y,B(x,r)^c)-\Mc{R}(x,y)}{2\Mc{R}(x,B(x,r)^c)}\\
		\ge&(C^{-1}r-\eta r)C^{-1}/2r\ge C^{-2}/4
	\end{align*}
for any $y\in B(x,\eta r).$ Hence let $g_{x,r}:=\bigl((4C^2f_{x,r}\wedge1)\vee 0\bigr),$ then $g_{x,r}|_{B(x,\eta r)}\equiv 1,\ g_{x,r}|_{B(x,r)^c}\equiv 0$ and $\Mc{E}(g_{x,r},g_{x,r})\le 16C^4\Mc{E}(f_{x,r},f_{x,r})\le 16C^3/r.$ This proves the statement for this case.\\
We now turn to the case $1/2C<\eta<1.$ Since $(X,R)$ is doubling, there exists $N=N_\eta\in\Mb{N}$ such that for any $x \in X,$ there exists $\{x_i\}_{i=1}^N\subset X$ satisfying $B(x,\eta r)\subset \cup_{i=1}^N B(x_i, (1-\eta)r/2C).$ Let $g=\max_{1\le i\le N}(g_{x_i,(1-\eta)r}),$ where $g_{x_i,(1-\eta)r}$ is same as above, then $g\in\Mc{F}$ by Lemma \ref{inf} (1). Moreover, $g|_{B(x,\eta r)}\equiv 1$ and $g|_{B(x,r)^c}\equiv0$ because $R(x_i,y)\ge R(x,y)-R(x,x_i)\ge (1-\eta)r$ for any $y\in B(x,r)^c$ and $1\le i\le N.$ Therefore
\begin{align*}
	\Mc{R}(\ol{B(x,\eta r)}, B(x,r)^c)^{-1}\le& \Mc{E}(g,g)\le \sum_{i=1}^N\Mc{E}(g_{x_i,(1-\eta)r},g_{x_i,(1-\eta)r})\\
	\le& 16N_\eta C^3/(1-\eta)r,
\end{align*}
which proves the lemma.
\end{proof}
\begin{cor}\label{bdr}
	Let $A_1, A_2\subset X$ be nonempty subsets. If $A_1$ is bounded and $R(A_1,A_2)>0,$ then $\Mc{R}(A_1,A_2)>0.$
\end{cor}
\begin{proof}
	Since $(X,R)$ is doubling, there exist $N\in\Mb{N}$ and $\{x_i\}_{i=1}^N\subset A_1$ with $A_1\subset \cup_{i=1}^N B(x_i,R(A_1,A_2)/2).$ Thus the proof is straightforward by Lemmas \ref{inf} (1) and \ref{resbb}.
\end{proof}

\begin{defi}[Local]\label{local}
	$(\Mc{E,F})$ is called \emph{local} if it satisfies $\Mc{E}(u,v)=0$ whenever $u,v\in \Mc{F}$ and $R(\{x\mid u(x)\ne0 \},\{x\mid v(x)\ne0 \})>0.$ 
\end{defi}
Under Assumption \ref{ass}, for each $\mu\in\Mc{M}_{(X,d)},$ $(\Mc{E,F})$ is a local resistance form if and only if $(\Mc{E}_\mu,\Mc{D}_\mu)$ is a local Dirichlet form. See Appendix \ref{ApLoc} for details.
\begin{prop}\label{Propwbd}
	Let $A_i\ (i=1,2)$ be nonempty subsets of $(X,R)$ with $R(A_1,A_2)>0$ and  $\diam(A_1)<\infty,$ $\{V_n\}_{n\ge0}$ be a spread sequence and $\mu_n$ be the resistance weights associated with $(\Mc{E}|_{V_n}, \ell(V_n))$. Assume $(\Mc{E,F})$ to be local, then
	\[\lim_{n\to\infty}\sum_{(x,y)\in D_n}(\mu_n)_{x,y}=0, \text{ where }D_n=(A_1\times A_2\cup A_2\times A_1)\cap V_n\times V_n. \]
	
\end{prop}

\begin{proof}
	Let \[A_1^*=\ol{\{x\mid R(x,A_1)\ge R(A_1,A_2)/3\}}\ \text{and}\ A_2^*=\ol{\{x\mid R(x,A_1)\le 2R(A_1,A_2)/3\}}.\]  By Corollary \ref{bdr}, we can take the optimal functions $f_i\in\Mc{F}$ for $\Mc{R}(A_i, A_i^*)$ and $i=1,2.$ Then,
	\begin{align*}
		&\Mc{E}|_{V_n}(f_1|_{V_n},f_1|_{V_n})+\Mc{E}|_{V_n}(f_2|_{V_n},f_2|_{V_n})-\Mc{E}|_{V_n}(f_1+f_2|_{V_n},f_1+f_2|_{V_n})\\
		=&\sum_{(x,y)\in D_n}(\mu_n)_{x,y}\\
		&\begin{aligned} +\frac{1}{2}\sum_{\substack{(x,y)\in V_n\times V_n\setminus D_n\\ 
					:x\ne y}}\bigl((f_1(x)&-f_1(y))^2+(f_2(x)-f_2(y))^2\\
				 &-((f_1+f_2)(x)-(f_1+f_2)(y))^2\bigr)(\mu_n)_{x,y}.
		\end{aligned}
	\end{align*} 
Since $0\le f_1,f_2 \le 1$ and $supp(f_1)\cap supp(f_2)=\emptyset,$ we have
\begin{align*}
	|(f_1+f_2)(x)-(f_1+f_2)(y)|=& |(f_1\vee f_2)(x)-(f_1\vee f_2)(y)|\\
	\le & |f_1(x)-f_1(y)|\vee |f_2(x)-f_2(y)|
\end{align*}	
for any $x,y\in X.$ Therefore
\begin{align*}
	0=&\Mc{E}(f_1,f_1)+\Mc{E}(f_2,f_2)-\Mc{E}(f_1+f_2,f_1+f_2)\\
	=&\lim_{n\to\infty}	\Mc{E}|_{V_n}(f_1|_{V_n},f_1|_{V_n})+\Mc{E}|_{V_n}(f_2|_{V_n},f_2|_{V_n})-\Mc{E}|_{V_n}(f_1+f_2|_{V_n},f_1+f_2|_{V_n})\\
	\ge& \lim_{n\to\infty}\sum_{(x,y)\in D_n}(\mu_n)_{x,y}\ge 0
\end{align*}
because $(\Mc{E,F})$ is local, which completes the proof.
\end{proof}

\begin{prop}\label{PropCo}
	There exists $\alpha>1$ with $R_{V, B(x, \alpha R(x,y))}(x,y)\le 2R(x,y)$ for any nonempty finite subset $V\subset X$ and $x,y\in V,$ where $\mu_{x,y}$ are the resistance weights associated with $(\Mc{E}|_V,\ell(V))$ and
	\[R_{V,A}(x,y)^{-1}=\min\{\frac{1}{2}\sum_{x,y\in A}(f(x)-f(y))^2\mu_{x,y}\mid f(x)=1,\ f(y)=0\}. \]
\end{prop}
\begin{rem}
	The idea and the proof of Proposition \ref{PropCo} essentially come from \cite[Lemma 2.5]{BCK}.
\end{rem}
\begin{proof}
	We first note that since Lemma \ref{resbb} holds,
	\[\Mc{R}_V(x,B(x,r)^c)\ge \Mc{R}_V(\ol{B(x,r/2)},B(x,r)^c)\ge  \Mc{R}(\ol{B(x,r/2)},B(x,r))^c\gtrsim r\]
	for any $x\in X$ and $r>0,$ where $\Mc{R}_V$ is the resistance between sets associated with $(\Mc{E}|_V,\ell(V)).$
	Thus we can find $\alpha>1$ such that
	\[\Mc{R}_V(x,B(x,(\alpha/2)r)^c)\vee \Mc{R}_V(\ol{B(x,(\alpha/2)r)},B(x,\alpha r)^c)\ge4r\]
	for any $x$ and $r.$  To shorten notation, we write $B_\beta$ instead of $B(x,\!B(x,\beta R(x,y))).$ Let $f_1$ (resp. $f_2$, $f_3$) be the optimal function for $R_{V, B_\alpha}(x,y)$ (resp. $\Mc{R}_V(x,B_{\alpha/2}^c),$ $\Mc{R}_V(\ol{B_{\alpha/2}},B_{\alpha}^c)$). We define $f\in\ell(V)$ by
	\begin{gather*}
		f(x)=\begin{cases}
			f_1(x)\wedge f_2(x)\wedge f_3(x) & (x\in B_\alpha)\\
			f_2(x)\wedge f_3(x) & (\text{otherwise})\end{cases}.\\
		\shortintertext{Then, $f(x)=1,$ $f(y)=0,$ $f|_{B_{2/\alpha}^c}\equiv 0$ and}
		|f(x)-f(y)|\le \begin{cases}
			\sum_{i=1}^3|f_i(x)-f_i(y)| &(\text{if }x,y\in B_{\alpha})\\
			0 &(\text{if }x,y\not\in B_{\alpha/2})\\
			|f_3(x)-f_3(y)|=1&
			\begin{aligned}(&\text{if }x\in B_{\alpha/2}\text{ and }y\not\in B_{\alpha}, \\
			&\text{or }y\in B_{\alpha/2}\text{ and }x\not\in B_{\alpha}).\end{aligned}
		\end{cases}
	\end{gather*}
	Therefore 
	\begin{align*}
		(R(x,y))^{-1}&\le\Mc{E}(f,f)\\
		&\le (R_{V, B_\alpha}(x,y))^{-1}+(\Mc{R}_V(x,B_{\alpha/2}^c))^{-1}+(\Mc{R}_V(\ol{B_{\alpha/2}},B_{\alpha}^c))^{-1}\\
		&\le(R_{V, B_\alpha}(x,y))^{-1}+\frac{1}{2}(R(x,y))^{-1}
	\end{align*}
	and $R_{V, B_\alpha}(x,y)\le 2R(x,y)$ as claimed.
\end{proof}
The remainder of this section is devoted to the proof of Proposition \ref{dsvh} below, which gives one of the key inequalities to prove Theorem \ref{main2}. For the rest of this section, we assume $d$ to be a metric on $X$ with $d\qs R.$ Then by Assumption \ref{ass}, it is easily shown that $\Mc{M}_{(X,d)}=\Mc{M}_{(X,R)}$ (see \cite[Corollary 12.4]{Kig12}, for example).  
\begin{defi}\label{rh}
	Set \[\ol{R}_d(x,r):=\sup_{y\in B_d(x,r)}R(x,y)\ \text{ and }\ h_{d,\mu}(x,r):=V_{d,\mu}(x,r)\ol{R}_d(x,r).\]
 We write $h_d(x,r)$ instead of $h_{d,\mu}(x,r)$ when no confusion can arise.
\end{defi}
\begin{prop}\label{dsvh}
	The limit $\ol{d_s}(\mu,\Mc{E}_\mu,\Mc{D}_\mu)$ exists for any $\mu\in\Mc{M}_{(X,d)}.$ Moreover,
	\begin{equation} \label{eqdsvh}
		\ol{d_s}(\mu,\Mc{E}_\mu,\Mc{D}_\mu)=2\limsup_{s\to\infty} \sup_{x\in X, r\in (0,\diam(X,d))}\frac{\log\bigl( V_{d,\mu}(x,r)/V_{d,\mu}(x,r/s)\bigr)}{\log\bigl( h_{d,\mu}(x,r)/h_{d,\mu}(x,r/s)\bigr)}.
	\end{equation}
\end{prop}
\begin{rem}
	We only use the case $d=R$ for the proof of Theorem \ref{main2} (recall that $R\qs R$). However, we prove general case for future works.
\end{rem}

We introduce some basic facts for the proof of Proposition \ref{dsvh}.
\begin{lem}\label{exp}
	Assume $\mu\in\Mc{M}_{(X,d)}$, then the following statements are true.
	\begin{enumerate}
		\item There exists $\gamma_1>1$ such that $V_d(x,r/\gamma_1)\le V_d(x,r)/2$ for any $x\in X$ and $r\in (0, \diam(X,d)).$
		\item $h_d(x,2r)\lesssim h_d(x,r)$ for any $x\in X$ and $r>0.$
		\item There exists $\gamma_2>1$ such that $h_d(x,r/\gamma_2)\le h_d(x,r)/2$ for any $x\in X$ and $r\in (0, \diam(X,d)).$
		\item For any $C>0,$ there exists $\gamma_C>1$ such that for any $t\in (0,C)$ and $x\in X,$ there exists $r\in (0,\diam(X,d))$ with $t/\gamma_C\le h_d(x,r)\le t$
		\item For any $x\in X,\ p_\mu(\cdot,x,x):t\mapsto p_\mu(t,x,x)$ is a decreasing function of $t.$
		\item Fix any $C'>0.$ Then $p_\mu(t/2,x,x)\lesssim p_\mu (t,x,x)$ for any $x\in X$ and $t\in (0, C').$
	\end{enumerate}
\end{lem}
\begin{proof}
	(1) It is well-known and easily follows from the volume doubling and uniformly perfect conditions (see \cite[Excersise 13.1]{Hei} for example).\par
	(2), (3) Since $d \qs R$ and both $d$ and $R$ are uniformly perfect, it is easy to check that there exists $\gamma'>1$ such that $\ol R_d(x,2r)\lesssim \ol R_d(x,r)$ and $\ol R_d(x,r/\gamma')\le \ol R_d(x,r)/2$ for any $x\in X$ and $r\in (0, \diam(X,d)).$ This with (1) and the volume doubling condition shows (2) and (3).\par
	(4) Since $C/\sup_{r\in (0,\diam(X,d))} h_d(x,r)\le 2C/\diam(X,d)\mu(X)$ for any $x\in X,$ it follows from (2) and (3).\par
	(5) By the proof of \cite[Theorem 10.4 and Lemma 10.7]{Kig12}, $\lim_{n\to\infty}p_n(t,x,y)=p_\mu(t,x,y)$ where $p_n(t,x,y): (0,\infty)\times X\times X\to \Mb{R}$ is of the form 
	\[p_n(t,x,y)=\sum_{k\ge1}\exp(-\lambda_{n,k}t)\varphi_{n,k}(x)\varphi_{n,k}(y)\]
	for some $\lambda_{n,k}>0$ and $\varphi_{n,k}: X\to\Mb{R}.$ Hence (5) is clear.\par
	(6) Recall that by Proposition \ref{pre}, $(X,R)$ satisfies (ACC). Thus this follows from (3)-(5), \cite[Theorem 15.6]{Kig12} and the fact that $p_\mu(t,x,x)\ge\mu(X)^{-1}$ for any $t>0,$ which follows from the Chapman-Kolmogorov equation.
\end{proof}
\begin{rem}
	The condition of Lemma \ref{exp} (1) is called reverse volume doubling condition (e.g. \cite{GHL, KM}).
\end{rem}

For the existence of $\ol{d_s}(\mu,\Mc{E}_\mu,\Mc{D}_\mu),$ we use the classical result for subadditive functions. For a proof, see \cite[Proof of Theorem 7.6.1]{HP} for instance.
\begin{lem}\label{sub}
	Let $f:(0,\infty)\to\Mb{R}$ be subadditive, that is, $f(t+s)\le f(t)+f(s)$ for any $t,s\in (0,\infty).$ Assume that $\sup_{t\in I} f(t)<\infty$ for any bounded interval $I,$ then $\lim_{t\to\infty}f(t)/t=\inf_{t>0}f(t)/t<\infty.$ 
\end{lem} 
\begin{proof}[Proof of Proposition \ref{dsvh}]
	Let 
	\[f(\tau)=\log\bigl( \sup_{x\in X, s\in(0,\diam(X,d))}p_\mu(s/e^\tau,x,x)/p_\mu(s,x,x)\bigr)\]
	for $\tau>0,$ then f is subadditive by definition. By Lemma \ref{exp} (5) and (6), $\sup_{\tau\in I}f(\tau)<\infty$ for any bounded interval $I$ and $\inf_{\tau>0}f(\tau)/\tau\ge 0.$ Hence $\lim_{\tau\to\infty}f(\tau)/\tau$ and the limit $\ol{d_s}(\mu,\Mc{E}_\mu,\Mc{D}_\mu)$ exist because $\tau\to\infty$ as $t=e^\tau\to\infty.$\par
	Our next goal is to prove \eqref{eqdsvh}. To this end, let 
	\begin{align*}
		u(s)=& \sup_{\substack{x\in X, \alpha\in [s,\infty),\\ r\in (0,\diam(X,d)) }}\frac{
			\log\bigl( p_\mu(h_d(x,r/\alpha),x,x)/ p_\mu(h_d(x,r),x,x)\bigr)}{
			\log\bigl( h_d(x,r) /h_d(x,r/\alpha)\bigr)}\text{ and}\\
		v(s)=& \sup_{\substack{x\in X, \alpha\in [s,\infty),\\ r\in (0,\diam(X,d)) }}\frac{
			\log\bigl( p_\mu(t/\alpha,x,x)/ p_\mu(t,x,x)\bigr)}{
			\log\alpha}.
	\end{align*} 
	By \cite[Theorem 15.6]{Kig12}, Proposition \ref{pre} and Lemma \ref{exp}(3), the right hand side of \eqref{eqdsvh} equals $\lim_{s\to\infty}u(s),$ hence it is sufficient to show $\lim_{s\to\infty}u(s)=\lim_{s\to\infty}v(s).$ We proceed to show the following claim.
	\begin{cla}
		For any $\epsilon>0,$ there exists $s_0(\epsilon)$ such that for any $s>s_0(\epsilon),$ there exists $s_*(s,\epsilon)$ satisfying $(1+\epsilon)(\epsilon+u(s))\ge v(s_*(s,\epsilon))$
	\end{cla}
 This claim implies $\lim_{s\to\infty}u(s)\ge\lim_{s\to\infty}v(s)$ because both $u$ and $v$ are decreasing.
 \begin{proof}
 For any $\epsilon>0$ and $s>1,$ we can take $C_1,...,C_4>1$ satisfying the following conditions by Lemma \ref{exp} (2), (4)-(6).
 \begin{itemize}
 	\item If $x\in X,\ r\in (0,\diam(X,d))$ and $\beta>0$ satisfy $h_d(x,r)/h_d(x,r/\beta)>C_1,$ then $\beta>s.$ 
 	\item For any $x\in X$ and $t\in (0,\diam(X,d)),$ there exists $r\in (0,\diam(X,d))$ with $t/C_2\le h_d(x,r)\le t.$
 	\item Any $x\in X$ and $t_1,t_2\in (0,\diam(X,d))$ with $t_1\le C_2t_2$ satisfy $C_3\ge \log (p_\mu(t_2,x,x)/p_\mu(t_1,x,x)).$
 	\item $C_4>C_2,$ $C_3/\log C_4<\epsilon$ and $\log C_4/(\log C_4- \log C_2)<1+\epsilon.$
 \end{itemize}  
Let $x\in X,\ t\in (0,\diam(X,d))$ and $\alpha>C_2(C_1\vee C_4)=:s_*$ We take $r_1, r_2\in (0,\diam(X,d))$ such that $t/C_2\le h_d(x,r_1)\le t\text{ and }t/C_2\alpha\le h_d(x,r_2)\le t/\alpha.$ Then $C_2\alpha>h_d(x,r_1)/h_d(x,r_2)>C_1\vee C_4$ and so for 
\[l_1:=\log \biggl( \frac{p_\mu(h_d(x,r_2),x,x) }{p_\mu(h_d(x,r_1),x,x)}\biggr)\quad \text{and}\quad l_2:=\log \biggl( \frac{p_\mu(h_d(x,r_2),x,x) }{p_\mu(h_d(x,r_1),x,x)}\biggr),\]
it follows that
\begin{align*}
		&\log\bigl(p_\mu(t/\alpha,x,x)/p_\mu(t,x,x) \bigr)/\log\alpha \\
	\le&\big(l_1+\log\bigl( p_\mu(h_d(x,r_1),x,x)/p_\mu(t,x,x) \bigr)\bigr)\big/\bigl(l_2-\log C_2\bigr) \\
	\le&\bigl((l_1/l_2)+(C_3/l_2))l_2\big/\bigl(l_2-\log C_2\bigr)\\
	\le&\bigl((l_1/l_2)+(C_3/\log C_4))\log C_4\big/\bigl(\log C_4-\log C_2\bigr).
\end{align*} 
Therefore $v(s_*)\le (1+\epsilon)(\epsilon+u(s))$ by $r_1/r_2>s,$ and the claim follows.
\end{proof}
For $\lim_{s\to\infty}u(s)\le\lim_{s\to\infty}v(s),$ we consider the case $\diam(X,d)<\infty$ and fix any $\epsilon>0.$ Then by Lemma \ref{exp} (5) and (6), there exists $s>1$ with
\[\sup_{\substack{x\in X,\\r\in(0,\diam(X,d))}}\log \biggl( \frac{p_\mu(\diam(X,d),x,x)}{p_\mu(h_d(x,r),x,x)}\biggr)<\epsilon \log s\]
because 
\begin{equation}\label{eq3}
	\sup_{x\in X, r\in (0,\diam(X,d))}\bigl( h_d(x,r)/\diam(X,d) \bigr)\le \mu(X)<\infty.
\end{equation} 
By \eqref{eq3} and Lemma \ref{exp} (3), we can take $s'>1$ such that if $\alpha>s'$ then 
$\bigl(h_d(x,r)\wedge\diam(X,d) \bigr)\big/ h_d(x,r/\alpha)>s$ for any $x\in X$ and $r\in (0,\diam(X,d)).$ Thus we obtain
\begin{align*}
	&\log \biggl( \frac{p_\mu(h_d(x,r/\alpha),x,x) }{p_\mu(h_d(x,r),x,x)}\biggr)\Bigg/ 
		\log \biggl( \frac{h_d(x,r)}{h_d(x,r/\alpha)}\biggr)\\
	\le&\Biggl(\log \biggl( \frac{p_\mu(h_d(x,r/\alpha),x,x) }{p_\mu(h_d(x,r)\wedge \diam(X,d),x,x)}\biggr)\Biggm/
		\log \biggl( \frac{h_d(x,r)\wedge\diam(X,d) }{h_d(x,r/\alpha)}\biggr)\Biggr)+\epsilon
\end{align*} 
and $u(s')<v(s)+\epsilon.$ This shows $\lim_{s\to\infty}u(s)\le\lim_{s\to\infty}v(s)$ in the same way as the inverse direction. The proof for the case $\diam(X,d)=\infty$ is similar, and the proposition follows.
\end{proof}
\section{Partition satisfying basic framework}\label{secpart}
In the former part of this section, we introduce the notion and related results of the \textit{partition satisfying the basic framework,} which is defined in \cite{Kig20} for the bounded case and extended to unbounded cases in \cite{Sas1}. In the latter part of this section, we show some related resistance estimates. Note that we continue to make Assumptions \ref{ass}.
\begin{defi}[Tree with a reference point]
	Let $T$ be a countable set and $\pi:T\to T$ be a map such that the following conditions hold.
	\begin{align*}
		\bullet\ & \text{Let }F_\pi=\{w\mid \pi^n(w)=w\text{ for some }n\ge1\},\text{ then }\# F_\pi\le1.  \\
		\bullet\ & \text{For any }w,v\in T,\text{ there exist }n,m\ge0\text{ such that }\pi^n(w)=\pi^m(v). 
	\end{align*}
	Let $\phi\in F_\pi$ if $F_\pi\ne\emptyset,$ otherwise we fix any $\phi\in T.$ We call the triplet $(T,\pi,\phi)$ a \textit{tree with a reference point}. 
\end{defi}
The above definition is justified as follows.
\begin{lem}[{\cite[Lemma 3.2]{Sas3}}]\hspace{1pt}\\ \vspace{-1\baselineskip} \begin{enumerate}
	\item Let $b(w,v)=\min\{n\ge0| \pi^n(w)=\pi^m(v)\text{ for some }m\ge0\}$ for $w,v\in T,$ then $\pi^{b(w,v)}(w)=\pi^{b(v,w)}(v).$
	\item Let $\Mc{A}=\{(w,v)\mid \pi(w)=v \text{ or }\pi(v)=w\}\setminus\{(\phi,\phi)\},$ then $(T,\Mc{A})$ is a tree.
\end{enumerate}
\end{lem}
From now on we assume $(T,\pi,\phi)$ to be a tree with a reference point. We define $[w]=b(w,\phi)-b(\phi,w), \ T_n=\{w\in T\mid [w]=n\}$ for any $w\in T$ and $n\in\Mb{Z}.$ By abuse of notation, we write $\acute\pi^{-k}(w)$ instead of $\pi^{-k}(w)\cap T_{[w]+k}.$ Note that $\acute\pi^{-k}(w)\ne\pi^{-k}(w)$ if and only if $F_\pi\ne\emptyset$, $w=\phi$ and $k\ge1.$ We also define $T^w=\cup_{k\ge0}\acute\pi^{-k}(w).$

\begin{defi}[Partition]\label{DefPart}
		Let $(Y,\rho)$ be a ($\sigma$-compact) metric space without isolated points. We say $K:T\to \Mf{P}(Y)$ is a \emph{partition} of $(Y,\rho)$ parametrized by $(T,\pi,\phi)$ if the following conditions hold.
	\begin{align*}
		\bullet\ & \text{For any }w\in T,\ K(w)\text{ is a compact set, neither a single point nor empty.}\\
		\bullet\ &\Cup_{w\in(T)_0}K(w)=Y\text{ and for any }w\in T,\ \Cup_{v\in\acute\pi^{-1}(w)}K(v)=K(w). \\
		\bullet\ &\text{If }(w_k)_{k\in\Mb{Z}}\subset T\text{ satisfies }\pi(w_{k+1})=w_k \text{ for any }k\in\Mb{Z},\text{ then}\cap_{k\in\Mb{Z}}K(w_k)\\
		&\text{is a single point.} \notag
	\end{align*}
Hereafter, we write $K_w$ instead of $K(w)$ for simplicity.
\end{defi}
\begin{rem}
	The condition that $K_w$ has no isolated points, assumed in \cite[Definition 2.2.1]{Kig20}, follows from Definition \ref{DefPart} (see \cite[Lemma 3.6]{Sas3}).
\end{rem}

\begin{defi}[Basic framework]\label{bf}
	Let $(T,\pi,\phi)$ be a tree with a reference point satisfying $\sup_{w\in T}\#(\acute\pi^{-1}(w))<\infty,$ and  $K$ be a partition of a metric space $(Y,\rho)$ parametrized by $(T,\pi,\phi).$ Let
	\[E_n=\{(w,v)\in T_n\times T_n\mid K_w\cap K_v\ne\emptyset,\ w\ne v\}\]
	and let $l_n$ denote the graph distance of $(T_n,E_n)$ allowing $l_n(w,v)=\infty.$ We say $K$ satisfies the \emph{basic framework} if the following conditions hold.
	\begin{align}
		\bullet\ & \nai(K_w)\cap\nai(K_v)=\emptyset\text{ for any }w,v\in T\text{ with }[w]=[v]\text{ and }w\ne v.\\
		\bullet\ & \text{There exists }\zeta\in(0,1)\text{ such that } \diam_\rho(K_w)\asymp \zeta^{[w]}\text{ for any }w\in T.\label{B1}\\
		\bullet\ & \text{There exists }\xi>0\text{ such that  for each }w\in T,\ B_\rho(x_w,\xi\zeta^{[w]})\subset K_w\notag\\ &\text{for some }x_w\in K_w.\label{B2}\\
		\bullet\ & \text{Let }\Delta_{m}(x,y)\!=\!\sup\{n\!\mid\! x\in K_w, y\in K_v \text{ and }l_n(w,v)\le m\text{ for some }w,v\in T_n\} \notag\\
		&\text{then }\rho(x,y)\asymp \zeta^{\Delta_{M_*}(x,y)}\text{ for any }x,y\in X, \text{ for some } M_*\in\Mb{N}. \label{B3}\\
		\bullet\ &\textstyle L_*:=\sup_{w\in T}\#(\{v\mid (w,v)\in E_{[w]} \})<\infty \label{B4}
	\end{align}
\end{defi}
\begin{rem}
	\begin{enumerate}
		\item The formulation in Definition \ref{bf} differs from the original one in \cite[Section 4.3]{Kig20}, for the reader's convenience. However it follows from \eqref{B1} and \cite[Proposition 3.2.1]{Kig20} that the above definition is equivalent to the original one.
		\item By \eqref{B1}, $\diam(X,d)<\infty$ if $\pi(\phi)=\phi$ and otherwise $\diam(X,d)=\infty.$
	\end{enumerate}
\end{rem}
For the existence of a partition of the
 given metric space satisfying the basic framework, there is the following result.
\begin{thm}[{\cite[Theorem 3.12]{Sas3}}]
	Let $(Y,\rho)$ be a complete metric space. Then the following conditions are equivalent.
\begin{enumerate}\label{xbf}
	\item $\ard(Y,\rho)<\infty.$
	\item $(Y,\rho)$ is doubling and uniformly perfect.
	\item There exist a tree with a reference point $(T,\pi,\phi)$ and a partition $K$ of $(Y,\rho)$ such that $K$ satisfies the basic framework with respect to $\rho.$
\end{enumerate}
\end{thm}
\begin{rem}
	\begin{enumerate}
		\item In \cite{Sas1,Sas2,Sas3}, the definition of the Ahlfors regular conformal dimension was slight different in order to consider that of a discrete metric space. This difference required the additional assumption that ($(Y,\rho)$ is) ``without isolated points'' in the original statement of  \cite[Theorem 3.12]{Sas3}.
		\item The equivalence between (1) and (2) was well-known (see \cite[Theorem 13.3 and Corollary 14.15]{Hei}, for example).
	\end{enumerate}
\end{rem}
We also note that we can choose $\{x_w\}_{w\in T_n}$ as an increasing sequence of sets.
\begin{lem}\label{partsp}
	Let $K$ be a partition of $(Y,\rho),$ parametrized by $(T,\pi,\phi)$ satisfying the basic framework. Then there exist $\{x_w\}_{w\in T}$ satisfying \eqref{B2} and for any $n\le m,$ $\cup_{w\in T_n}\{x_w\}\subset \cup_{w\in T_m}\{x_w\}.$ 
\end{lem}
\begin{rem}
	It is obvious that $\{x_w\mid w\in T_n\cap T^{\pi^n(\phi)}\}_{n\ge0}$ is a spread sequence.
\end{rem}
\begin{proof}
	Let $\{x_w\}_{w\in T}$ be points satisfying \eqref{B2}. By \eqref{B1}, there exists $k\ge1$ such that  $g_\rho(K_w)\le \xi\zeta^n/2$ for any $w\in T_{n+k}.$ We can define $f:\cup_{n}T_{kn}\to\cup_{n}T_{kn}$ such that $f(w)\in T_{[w]+k}$ and $x_w\in K_{f(w)}.$ For $w\in\cup_{n}T_{kn},$ let $y_w$ be the unique point with $y_w\in\cap_{n\ge0}K_{f^n(w)}.$ Then $\cup_{w\in T_{kn}}\{y_w\}\subset \cup_{w\in T_{km}}\{y_w\}$ for $n\le m$ and
	\[B_\rho(y_w, \frac{\xi}{2}\zeta^{[w]})\subset B_\rho(x_w, \xi\zeta^{[w]})\subset K_w\]
	for $w\in\cup_{n}T_{kn}.$ For $w\in \cup_{n}T_{kn-m}\ (k>m>0),$ we define $y_w$ by induction on $m.$ Let $y_w=y_v$ for some $v\in\acute\pi^{-1}(w)$ such that $v=\pi^{m-1}\circ f\circ \pi^{k-m}(w)$ whenever $w=\pi^m\circ f\circ \pi^{k-m}(w).$  Then  we obtain
	\begin{gather*}
		\cup_{w\in T_{kn-(m-1)}}\{y_w\}\supset \cup_{w\in T_{kn-m}}\{y_w\}\supset \cup_{w\in T_{k(n-1)}}\{y_w\}
		\shortintertext{and}
		B_\rho(y_w, \frac{\xi\zeta^k}{2}\zeta^{[w]})\subset B_\rho(y_v, \frac{\xi}{2}\zeta^{[v]})\subset K_v\subset K_w
	\end{gather*}
	for some $v\in\acute\pi^{-m}(w).$ This shows $\{y_w\}_{w\in T}$ is the desired set of points.
\end{proof}
 We are now able to introduce the precise definitions of $\ol{d}^s_p$ and $\ul{d}^s_p.$ For the rest of this section, we assume that $K$ denotes a partition of a metric space $(Y,\rho)$ parametrized by $(T,\pi,\phi)$ satisfying the basic framework, with points $\{x_w\}_{w\in T}$ satisfying \eqref{B2} such that for any $n\le m,$ $\cup_{w\in T_n}\{x_w\}\subset \cup_{w\in T_m}\{x_w\}.$
 \begin{defi}[$p$-spectral dimensions]\label{pspec}
 Let
 \begin{align*}
 	N_*&=\limsup_{k\to\infty}\Bigl(\sup_{w\in T}\#(\acute\pi^{-k}(w))\Bigr)^{1/k},\\
 	\Mc{E}^p_n(f)&=\frac{1}{2}\sum_{(x,y)\in E_n}|f(x)-f(y)|^p,\\
 	\Mc{C}_{w,k}&=\{v\in T_{[w]+k}\mid  l_{[w]}(w,\pi^k(v))>M_*\}\text { and}\\
 	\Mc{E}_{p,k,w}&=\inf\{\Mc{E}^p_n(f)\mid f\in\ell(T_{[w]+k}),\ f|_{\acute\pi^{-k}(w)}\equiv1,\ f|_{\Mc{C}_{w,k}}\equiv 0\}
 \end{align*}
 for any $n,$ $f\in\ell(T_n),$ $w\in T$ and $p>0.$ (In particular for $p=2,$  $\Mc{E}_{2,k,w}=(\Omega_{[w]+k}(\acute\pi^{-k}(w),\Mc{C}_{w,k}))^{-1}$ where $\Omega_n$ is the standard graph resistance of $(T_n,E_n)$.) We define the \emph{upper $p$-spectral dimensions} of the partition $K$ for $p>0$ by
 \begin{equation}\label{uds}
 	 \ol{d}^s_p(K)=p\biggl(1-\frac{\limsup_{k\to\infty}\frac{1}{k}\bigl(\sup_{w\in T}\log\Mc{E}_{p,k,w}\bigr) }{\log N_*}\biggr)^{-1}
 \end{equation}
and the \emph{lower $p$-spectral dimensions} $\ul{d}^s_p(K)$ for $p>0$ by \eqref{uds} but replacing $\limsup$ by $\liminf.$
 \end{defi}
 \begin{rem} 
 	In the same way as in the proof of Proposition \ref{dsvh}, we have 
 	\begin{equation}\label{ninf}
 	N_*=\lim_{k\to\infty}\bigl(\sup_{w\in T}\#(\acute\pi^{-k}(w))\bigr)=\inf _{k\ge0}\bigl(\sup_{w\in T}\#(\acute\pi^{-k}(w))\bigr)
 	\end{equation}
 	because $\bigl(\sup_{w\in T}\#(\acute\pi^{-j}(w))\bigr)\bigl(\sup_{w\in T}\#(\acute\pi^{-k}(w))\bigr)\ge\bigl(\sup_{w\in T}\#(\acute\pi^{-(j+k)}(w))\bigr)$ for any $j,k\ge0.$
 	
 \end{rem}
The following is the main result of \cite{Kig20}, which leads to Theorem \ref{Kmain}.
\begin{thm}[{\cite[Theorems 4.6.9]{Kig20} and \cite[Theorem 3.9]{Sas1}}]
	\begin{align*}
		\ard(Y,\rho)&=\inf\{p\mid\liminf_{k\to\infty}(\sup_{w\in T}\Mc{E}_{p,k,w})=0 \}\\ &=\inf\{p\mid\limsup_{k\to\infty}(\sup_{w\in T}\Mc{E}_{p,k,w})=0 \}.
	\end{align*}

\end{thm}
 In the reminder of this section, we assume $(Y,\rho)=(X,R)$ and prove some inequalities of indices of the partition, which are necessary for the proof of Theorem \ref{main2}. \par Let $V_n=\{x_w\mid w\in T_n\}$ and $A_w=\cup\{K_v\mid v\in \Mc{C}_{w,0}\}.$ We will denote by $\Mc{R}_n$ (resp. $R_n,\ \mu_{x,y,n}$) the resistance between sets (resp. metric, weights) associated with $(\Mc{E}|_{V_n},\ell(V_n)).$ 
\begin{lem}\label{resaw}
	$\Mc{R}(K_w, A_w)\asymp \zeta^{[w]}$ for any $w\in T$ with $A_w\ne\emptyset.$
\end{lem}
\begin{proof}
	Since \eqref{B1} and \eqref{B2} hold, $\Mc{R}(K_w,A_w)\lesssim \zeta^{[w]}$ follows from Lemma \ref{resbb}. On the other hand, there exists $\iota\in(0,1)$ satisfying $R(K_w, A_w)>\iota\zeta^{[w]}$ for any $w\in T$ because \eqref{B3} holds. Since $(X,R)$ is doubling each $K_w$ is covered by $N$ balls of radius $\iota\zeta^{[w]}/2,$ for some $N\ge0.$  Therefore $\Mc{R}(K_w,A_w)\gtrsim \zeta^{[w]}$ by Lemmas \ref{inf} (1) and \ref{resbb}, similarly to the latter part of the proof of Lemma \ref{resbb}.
\end{proof}
\begin{prop}\label{e2kw}
	\begin{enumerate}
		\item $\Mc{E}_{2,k,w}\lesssim \zeta^k$ for any $w\in T.$
		\item  If $(\Mc{E,F})$ is local then $\Mc{E}_{2,k,w}\gtrsim \zeta^k$ for any $w\in T.$
	\end{enumerate}
\end{prop}
For proving Proposition \ref{e2kw} (1), we use the argument of flow.

\begin{defi}[Unit flow]
	Let $(\Mc{E,F})$ be a resistance form on a finite set $V.$ For $A,B\subset V$ with $A\cap B=\emptyset,$ $f:V\times V\to R$ is called a \emph{unit flow} from $A$ to $B$ if it satisfies
	\begin{itemize}
		\item $f(x,y)=-f(y,x)$ for any $x,y\in V.$
		\item $\sum_{y\in V}f(x,y)=0$ for any $x\not\in A\cup B,$
		\item $\sum_{x\in A}\sum_{y \in V}f(x,y)=1$ and $\sum_{x\in B} \sum_{y\in V}f(x,y)=-1.$
	\end{itemize}
	Let $\{\mu_{x,y}\}, \Mc{R}$ be the associated resistance weight and resistance between subsets, then it is known that
	\begin{equation}\label{fl}
		\Mc{R}(A,B)=\min\{\frac{1}{2}\sum_{x,y\in V}\frac{f(x,y)^2}{\mu_{x,y}}\mid f\text{ is a unit flow from }A\text{ to }B \}.
	\end{equation}
	
	We say $f$ is the \emph{optimal flow} for $\Mc{R}(A,B)$, or optimal flow from $A$ to $B$ if $f$ is the optimal function for the right hand side of \eqref{fl}.
\end{defi}
\begin{proof}[Proof of Proposition \ref{e2kw}](1) Fix any $w\in T$ and $k\ge0.$ Let $\tau$ be the optimal flow for $\Omega_{[w]+k}(\acute\pi^{-k}(w),\Mc{C}_{w,k}),$ $\alpha>1$ be the constant appeared in Proposition \ref{PropCo} and  $f_{u,v}$ be the optimal flow for $R_{V_{[w]+k},B(x,\alpha R(x_u,x_v))}(x_u,x_v)$ for $u,v\in T_{[w]+k}.$
We define $f:V_{[w]+k}\times V_{[w]+k}\to \Mb{R}$ by
\[f(p,q)=\frac{1}{2}\sum_{(u,v)\in E_{[w]+k}}\tau(u,v)f_{u,v}(p,q),\]
then $f$ is a  unit flow from $K_w\cap V_{[w]+k}$ to $A_w\cap V_{[w]+k}$ on $(\Mc{E}|_{V_n},\ell(V_n))$. Note that $f_{u,v}(p,q)=0$ if 
\[R(x_u,p)\vee R(x_v,q)\ge \alpha R(x_u,x_v)\ge \alpha\xi \zeta^{[w]+k}.\] 
Since $R$ is doubling, there exists $N>0$ such that 
\[\sup_{x\in X,\ r>0} \{\#(Y) \mid Y\subset B(x,\alpha r), R(y,z)\ge r\text{ for any }y,z\in Y\text{ with }y\ne z \} \le N\]
Therefore
\begin{align*}
	\Mc{R}_{[w]+k}(K_w,A_w)\le&\frac{1}{2}\sum_{p,q\in V_{[w]+k}}\frac{f(p,q)^2}{\mu_{p,q}}\\
	\le& \frac{N}{8}\sum_{(u,v)\in E_{[w]+k}} \tau(u,v)^2\sum_{p,q\in V_{[w]+k}}\frac{f_{u,v}(p,q)^2}{\mu_{p,q}}\\
	\le& \frac{N}{2} \Bigl(\sup_{(u,v)\in E_{[w]+k}} 2R(x_u,x_v)\Bigr) \sum_{(u,v)\in E_{[w]+k}} \tau(u,v)^2\\
	\le& 2N\xi\zeta^{[w]+k}(\Mc{E}_{2,k,w})^{-1}.
\end{align*}
Since $\zeta^{[w]}\lesssim\Mc{R}(K_w,A_w)\le\Mc{R}_{[k]+w}(K_w,A_w)$ for any $w\in T,$ the claim holds.
	\par (2) We first note that there exists $\beta>0$ such that $\Delta_{M_*} (x,y)\ge n$ whenever $R(x,y)\le \beta \zeta^n$ by \eqref{B3}. Fix any $w\in T$ and $k\ge0.$ Since $T_{[w]+k}\setminus \Mc{C}_{w,k}$ is a finite set, there exists $n\ge[w]+k$ such that
	\[\sum_{v\in T_{[w]+k}\setminus \Mc{C}_{w,k}}\sum_{\substack{x\in K_v,\\ y\not\in B(x,\beta\zeta^{[w]+k})}}\mu_{x,y,n}<\frac{1}{3}\Mc{R}(K_w,A_w)^{-1}\]
	and $\Mc{R}_n(K_w,A_w)\le 2\Mc{R}(K_w,A_w)$, by Lemma \ref{resab} and Proposition \ref{Propwbd}.\\
	Let $f$ be the optimal function for $\Omega_{[w]+k}(\acute\pi^{-k}(w),\Mc{C}_{w,k})$ and $\tau_v$ be the optimal functions for $R_{n}(K_v,A_v)$ and $v\in T_{[w]+k}.$ We also let 
	\[\ol{f}(v)=2\max\bigl\{|f(v)-f(u)|\bigm\vert l_{[w]+k}(u,v)\le 2M_*\bigr\}.\]
	Our next goal is to construct a suitable function $\tau$ on $V_n$ with $\tau|_{K_w\cap V_n}\equiv 1$ and $\tau|_{A_w\cap V_n}\equiv 0$ with the above functions. Set
	\begin{gather*}
		\begin{multlined}
			P(x)=\bigl\{ \{x_i\}_{i=0}^m\subset V_n \bigm\vert m\in\{0\}\cup\Mb{N},\ x_0\in A_w,\ x_m=x,\bigr. \\
			\bigl.\quad R(x_i,x_{i-1})\le \beta\zeta^{[w]+k}\text{ for any }i\bigr\}
		\end{multlined}
		\\
		\shortintertext{ for $x\in V_n.$ We define $\tau:V_n\to\Mb{R}$ as}
		\tau(x)=1\wedge \inf \biggl\{\sum_{i=1}^m\sup_{v\in T_{[w]+k}} \ol{f}(v)|\tau_v(x_i)-\tau_v(x_{i-1})|
		\biggm\vert  \{x_i\}_{i=0}^m\in P(x)\biggr\}
	\end{gather*}
where $\sum_{i=1}^0\infty:=0.$ Then it clearly holds $0\le\tau\le1$ and $\tau|_{A_w\cap V_n}\equiv0.$ 
\begin{cla}
	$\tau|_{K_w\cap V_n}\equiv1$.
\end{cla}
\begin{proof}
	Fix any  $x\in K_w\cap V_n$ and $\{x_i\}_{i=0}^m\in P(x).$ We inductively choose
$i_0=0,$ $v_j\in T_{[w]+k}$ such that $x_{i_j}\in K_{v_j},$ and $i_{j+1}=\min\{i\mid x_j\in A_{v_j} \}.$ Note that since $\Delta_{M_*}(x_{i_j},x_{(i_j)-1})\ge [w]+k,$ $l_{[w]+k}(v_j,v_{j-1})\le 2M_*.$ Let $\iota\ge0$ and $v_*\in T_{[w]+k}$ be such that $x\not\in A_{v_\iota},$ $v_*\in T_{[w]+k}$ and $l_{[w]+k}(v_*,v_\iota)\le M_*,$ then
\begin{align*}
	&\sum_{i=1}^m\sup_{v\in T_{[w]+k}} \ol{f}(v)|\tau_v(x_i)-\tau(x_{i-1})|\\
	\ge& \sum_{j=0}^{\iota-1} \ol{f}(v_j)\sum_{i=i_j}^{i_{(j-1)}-1}|\tau_{v_j}(x_{i+1})-\tau(x_i)|\\ \ge&\sum_{j=0}^{\iota-1} \ol{f}(v_j)
	\ge|f(v_*)-f(v_\iota)|+\sum_{j=0}^{\iota-1}|f(v_{i+1})-f(v_i)|\ge1.
\end{align*} 
This shows the claim.
\end{proof}
Next we evaluate $\Mc{E}|_{V_n}(\tau,\tau):$ 
\begin{align*}
	\frac{1}{2}(\Mc{R}(K_w,A_w))^{-1}\le& \Mc{E}|_{V_n}(\tau,\tau)\\ \le& \frac{1}{2}\sum_{\substack{x,y\in V_n\\: (x,y)\not\in A_w\times A_w}}(\tau(x)-\tau(y))^2\mu_{x,y,n}\\
	\le&\frac{1}{3}(\Mc{R}(K_w,A_w))^{-1}+\frac{1}{2}\sum_{\substack{x,y\in V_n\\:R(x,y)\le \beta\zeta^{[w]+k}}}(\tau(x)-\tau(y))^2\mu_{x,y,n}.
\end{align*}
If $R(x,y)\le \beta \zeta^{[w]+k},$ then $|\tau(x)-\tau(y)|\le\sup_{v\in T_{[w]+k}}\ol{f}(v)|\tau_v(x)-\tau_v(y)|.$ Moreover, since $ \tau|_{A_v}\equiv0,$ $R(x_v,x_u)\ge \xi\zeta^{[w]+k}$ for $u,v\in T_{[w]+k}$ with $u\ne v$ and $R$ is doubling, there exists $J>0$ such that 
\[\sup_{\substack{w\in T, k\ge0,\\x,y\in V_n:R(x,y)\le \beta\zeta^{[w]+k}}}\#\{v\in T_{[w]+k}\mid \tau_v(x)\vee \tau_v(y)\ge0 \}\le J<\infty.\]
Therefore 
\begin{align*}
\frac{1}{6}(\Mc{R}(K_w,A_w))^{-1}\le& \frac{J}{2}\sum_{v\in T_{[w]+k}}\ol{f}(v)^2\sum_{x,y\in V_n}(\tau_v(x)-\tau_v(y))^2\mu_{x,y,n}\\
\le & \frac{J}{\xi}\zeta^{-[w]-k}\sum_{v\in T_{[w]+k}}\ol{f}(v)^2.
\end{align*}
Since $\ol{f}(v)\le \sum_{i=1}^{2M_*}|f(v_i)-f(v_{i-1})|$ for some $\{v_i\}_{i=0}^{2M_*}$ satisfying $v_0=v$ and $(v_i,v_{i-1})\in E_{[w]+k},$ 
\[\sum_{v\in T_{[w]+k}}\ol{f}(v)^2\le 4L_*^{2M_*-1}\sum_{u,v\in E_{[w]+k}}(f(u)-f(v))^2=8L_*^{2M_*-1}\Mc{E}_{2,k,w}.\]
Thus we have
\[ \Mc{E}_{2,k,w}\ge \frac{\xi}{48JL_*^{2M_*-1}}\zeta^{[w]+k}(\Mc{R}(K_w,A_w))^{-1}.\]
This with Lemma \ref{resaw} shows the proposition.
\end{proof}
\begin{prop}\label{nstar}
	\begin{align*}
		(1)\quad &\sup_{\substack {x\in X, r\in (0,\diam(X,R))}}\frac{V_\mu(x,r)}{V_\mu(x,\zeta^kr)}\gtrsim N_*^k\text{ for any }\mu\in \Mc{M}_{(X,R)}\text{ and }k\ge0.\\
		(2)\quad& \text{ For any }\epsilon>0, \text{ there exists } \mu\in\Mc{M}_{(X,R)}\text{ with }\\ &\sup_{\substack {x\in X, r\in (0,\diam(X,R))}}\frac{V_\mu(x,r)}{V_\mu(x,\zeta^kr)}\lesssim (N_*+\epsilon)^k\text{ for any }k\ge0.
	\end{align*}
\end{prop}
\begin{proof}
(1) We have 
\[\mu(K_w)\ge\sum_{v\in\acute\pi^{-k}(w)}V_\mu(x_v,\xi\zeta^k)\ge\min_{v\in\acute\pi^{-k}(w)}V_\mu (x_v,\xi\zeta^k)\#(\acute\pi^{-k(w)}),\]
which leads to
\[\sup_{\substack {w\in T,\\ v\in \acute\pi^{-k}(w)}}\frac{V_\mu(x_v,2\diam(K_w))}{V_\mu(x_v,\xi\zeta^{[v]})}\ge\sup_{w\in T}\#(\acute\pi^{-k(w)})\ge N_*^k\]
for any $k\ge0.$ Since $\mu$ is (VD)$_R,\ \diam(K_w)\lesssim\zeta^{[w]}$ for any $w\in T$ and \eqref{ninf}, the claim follows.\par
(2) Fix any $\epsilon>0.$ By \eqref{ninf} and \cite[Proposition 4.3.5]{Kig20}, we can choose $k\ge1$ such that $\sup_{w\in T}\#(\acute\pi^{-k}(w))\le (N_*+\epsilon)^k$ and for any $w\in T,$ there exists $v(w)\in \acute\pi^{-k}(w)$ with $K_{v(w)}\subset \nai(K_w).$\\
Let $\tilde{T}=\cup_{n\in \Mb{Z}}T_{kn},$ then it is easily seen that $K|_{\tilde T}$ is a partition of $(X,R)$ parametrized by $(\tilde T, \pi^k, \phi),$ satisfying the basic framework. We define $\varphi,\psi:\tilde T\to\Mb{R}$ as
\[
\varphi(w)=\begin{cases}
	\Bigl(1- \frac{\#(\acute\pi^{-k}(\pi^k(w)))-1}{(N_*+\epsilon)^k}\Bigr) & (\text{if }w=v(\pi^k(w)))\\
	\frac{1}{(N_*+\epsilon)^k} & \text{(otherwise)}
\end{cases}\]
and $\psi(w)=\bigl(\prod_{i=0}^{\tilde b(w,\phi)}\varphi (\pi^{ik}(w))\bigr)/\bigl(\prod_{i=0}^{\tilde b(\phi,w)}\varphi (\pi^{ik}(\phi))\bigr),$ where
\[\tilde b(w,u)=b_{\pi^k}(w,u):=\min\{i\ge0\mid \pi^{ik}(w)\text{ for some }j\ge0\}
\] for $w,u\in\tilde T.$ Note that 
\[\frac{1}{(N_*+\epsilon)^k}\le \varphi(w)\le  \Bigl(1- \frac{1}{(N_*+\epsilon)^k}\Bigr)\quad \text{and }\sum_{v\in\acute\pi^{-k}(w)}\varphi(w)=1.\]
We next claim that 
\begin{equation}
	((N_*+\epsilon)^k-1)\psi(w)\ge\max\{\psi(u)\mid (w,u)\in E_{[w]}\}\text{ for any }w\in\tilde{T}.\label{psi} 
\end{equation}
Recall that $\tilde b(w,u)=\tilde b(u,w)$ by $[w]=[u].$ If $i<\tilde b(w,u)-1$ then we have $\pi^{ik}(w)\ne v(\pi^{(i+1)k}(w))$ because
\[\phi\ne K_{\pi^{ik}(w)} \cap K_{\pi^{ik}(u)}\subset K_{\pi^{(i+1)k}(w)} \cap K_{\pi^{(i+1)k}(u)}\not\subset \nai( K_{\pi^{(i+1)k}(w)}). \]
This shows 
\[ \frac{\psi(u)}{\psi(w)}=\frac{\prod_{i=0}^{\tilde b(w,u)}\varphi (\pi^{ik}(w))}{\prod_{i=0}^{\tilde b(u,w)}\varphi (\pi^{ik}(u))}=\frac{\pi^{(\tilde b(w,u)-1)k}(w)}{\pi^{(\tilde b(u,w)-1)k}(u)}\le ((N_*+\epsilon)^k-1).\] 
We now prove the proposition for the case of $\diam(X,R)<\infty.$\\
Let $\mu_n=\sum_{w\in T_{nk}}\psi (w)\delta_{x_w},$ where $\delta_{x_w}$ is the Dirac measure on $x_w.$ Then by Prokhorov's theorem, there exists a Borel probability measure $\mu_*$ such that $\mu_{n_m}\to \mu_*$ weakly as $m\to\infty,$ for some subsequence $\{\mu_{n_m}\}_{m\ge0}.$ We have
\[\mu_*(K_w)\ge \limsup_{m\to\infty}\mu_{n_m}(K_w)=\psi(w)\]
for any $w\in T.$ Moreover, since $K_w\subset U_w:=\nai(\cup_{u:(w,u)\in E_{[w]}}K_u),$
\[\mu_*(K_w)\le \mu_*(U_w)\le\liminf_{m\to\infty}\mu_{n_m}(U_w)\le (L_*((N_*+\epsilon)^k-1)+1)\psi(w).\]
By \eqref{B1} and \eqref{B3}, there exists $m\ge0$ such that for any $x\!\in\! X,$ $r \!\in\!(0,\diam(X,R))$ and $n\ge0,$ we can choose some $w\in \tilde T$ with
\[K_w\subset B(x,\zeta^{nk}r)\quad\text{and}\quad B(x,r)\subset X\setminus A_{\pi^{(m+n)k}(w)}.\]
Since 
\begin{align*}\frac{\mu_*(X\setminus A_{\pi^{(m+n)k}(w)})}{\mu_*(K_w)}\lesssim &\frac{\sum^{u\in T_{[w]-(m+n)k}}_{:l_{[u]}(u,\pi^{(m+n)k}) \le M_*}\psi(u)}{\psi(w)}\\ \lesssim& \frac{\psi(\pi^{(m+n)k}(w))}{\psi(w)}\le(N_*+\epsilon)^{(m+n)k} \end{align*}
for any $n\ge 0$ and $w\in\tilde T$ by \eqref{psi}, the claim holds for the bounded case.\par
We now turn to the case of $\diam(X,R)=\infty.$ We can choose $\mu_u$ for each $u\in T_0$ such that $\mu_u(X\setminus K_u)=0$ and $\psi(w)\le \mu_u(K_w)\le L_*(N_*+\epsilon)^k\psi(w)$ for any $w\in \tilde T^u,$ by the former case. Let $\mu_*=\sum_{u\in T_0}\mu_u,$ then it is clear that $\psi(w)\le \mu_*(K_w)$ and
\begin{align*}
	\mu_*(K_w)\le&\begin{cases}
		\sum_{u:(u,w)\in E_{[w]}}\sum_{q\in T^u\cap T_0}\mu_q(K_u) & (\text{if } [w]<0)\\
		\sum_{u:(u,w)\in E_{[w]}}\mu_{\pi^{[w]}(u)}(K_u) & (\text{otherwise } )
	\end{cases} \\
\le& \sum_{u:(u,w)\in E_{[w]}}L_*(N_*+\epsilon)^k\psi(u)\le L_*^2(N_*+\epsilon)^{2k}\psi(w)
\end{align*}
for any $w\in \tilde T.$ Therefore the same proof as the bounded case works for the present case. 
\end{proof}
\begin{rem}
	\begin{enumerate}
		\item The idea of this proof comes from \cite[Theorems 4.2.2 and 4.5.1]{Kig20}.
		\item We can show that $-\log N_*/\log \zeta$ coincides with the \textit{Assouad dimension} $\dim_A(X,R),$ where
		\begin{multline*}
				\dim_A(X,R)=\inf\{t>0\mid\text{$B(x,r)$ is covered by}\lfloor C(r/s)^t\rfloor\text{ bolls}\\ \text{of radius $s$ for any $0<s<r$ and $x\in X$, for some $C>0$} \}.
		\end{multline*}
	Thus we can also deduce Proposition 3.12 (2) from \cite[Theorem 13.5]{Hei}.
	\end{enumerate}
\end{rem}
\section{Proof of main results}\label{secpf}
In this section we prove Theorems \ref{main} and \ref{main2}.
\begin{proof}[Proof of Theorem \ref{main2}]
	By Proposition \ref{dsvh}, 
	\begin{align}
		&\frac{\ol{d_s}(\mu,\Mc{E}_\mu,\Mc{D}_\mu)}{2}\notag\\
		=&\limsup_{s\to\infty} \sup_{x\in X, r\in (0,\diam(X,R))}\biggl(1+
		\frac{\log s}{\log\bigl( V_\mu(x,r)/V_\mu(x,r/s)\bigr)}\biggl)^{-1}\notag\\
		=&\limsup_{k\to\infty} \sup_{x\in X, r\in (0,\diam(X,R))}\biggl(1+
		\frac{\log \zeta^{-k}}{\log\bigl( V_\mu(x,r)/V_\mu(x,\zeta^kr)\bigr)}\biggl)^{-1}\notag\\
		=&\biggl(1+
		\frac{\log \zeta^{-1}}{\limsup_{k\to\infty} \sup_{x\in X, r\in (0,\diam(X,R))}\frac{1}{k}\log\bigl( V_\mu(x,r)/V_\mu(x,\zeta^kr)\bigr)}\biggl)^{-1}\label{eqds}
	\end{align}
because $R$ is uniformly perfect and $\mu$ is (VD)$_R.$\par Since $\limsup_{k\to\infty}(\sup_{w}\Mc{E}_{2,k,w})^{1/k}\le\zeta<1$ by Proposition \ref{e2kw} (1), 
\begin{equation}
	\frac{\ol{d}^s_2(K)}{2}\le \biggl(1+\frac{\log \zeta^{-1}}{\log N_*}\biggr)^{-1},\label{it}
\end{equation}	
 so $\ol{d}^s_2(K)\le \ol{d_s}(\mu,\Mc{E}_\mu,\Mc{D}_\mu)$ holds by Proposition \ref{nstar} (1). Moreover, if $(\Mc{E,F})$ is local, the equality in \eqref{it} holds by Proposition \ref{dsvh}. Since Proposition \ref{nstar} (2) holds, for any $\epsilon>0$ there exists $\mu\in\Mc{M}_{(X,R)}$ such that 
 \[\frac{\ol{d_s}(\mu,\Mc{E}_\mu,\Mc{D}_\mu)}{2}\le \biggl(1+\frac{\log \zeta^{-1}}{\log (N_*+\epsilon)}\biggr)^{-1},\]
 which shows $\inf_{\mu\in\Mc{M}_{(X,R)}}\ol{d_s}(\mu,\Mc{E}_\mu,\Mc{D}_\mu)\le \ol{d}^s_2(K)$ for the local case. 
\end{proof}
\begin{proof}[Proof of Theorem \ref{main}]
	We have a partition of $(X,R)$ satisfying the basic framework by Theorem \ref{xbf}, and obtain $\ol{d}^s_2(K)\le\ol{d_s}(\mu,\Mc{E}_\mu,\Mc{D}_\mu)<2$ by \eqref{eqds} and Theorem \ref{main2}. This and Theorem \ref{Kmain} (1) with $p=2$ prove the theorem.
\end{proof}

\section{Example with \texorpdfstring{$d_s(\mu,\Mc{E}_\mu,\Mc{D}_\mu)<\ard(X,R)<2$} {ds((mu),E(mu),D(mu))<dimARC(X,R)<2}}\label{seccex}
In this section we prove Theorem \ref{cex}. In other words, we give an example with the inequality $d_s(\mu,\Mc{E}_\mu,\Mc{D}_\mu)<\ard(X,R)<2.$ The results in this section are the continuous version of \cite[Section 3]{Sas2}, and many of the resistance estimates used in this section come from that preprint. We also note that the techniques used for showing these resistance inequalities was originated in \cite{BB90}.
The main difficulty in the continuous case is to construct the desired resistance form. We overcome this difficulty by using the results of \cite{Kas,Kig01} in the proof of Corollary \ref{lim}.\par
Recall that  $Q=\{z\mid |\mre (z)|\vee |\mim (z)| \le 1/2\}.$ Let
\begin{align*}
	p_i=&\begin{cases}
		0&(i=0)\\
		\frac{1}{\sqrt{2}}\bigl( \frac{1+\sqrt{-1}}{\sqrt{2}}\bigr)^i & (i=1,3,5,7)\\
		\frac{1}{2}\bigl( \sqrt{-1}\bigr)^{i/2} & (i=2,4,6,8)
	\end{cases}, \hspace{20pt} \varphi_i(z)=\frac{1}{3}(z-p_i)+z,\\
F(n)=&\begin{cases}
		1 & (\text{if }k^2(k-1)<n\le k^3\text{ for some }k\in\Mb{N})\\
		0 & (\text{otherwise})
\end{cases},\\
\Phi_n(\Mc{S})=&\begin{cases}
	\Cup_{i=1}^8 \varphi_i(\Mc{S}) & (\text{if }F(n)=1)\\
	\Cup_{i=0,1,3,5,7} \varphi_i(\Mc{S}) & (\text{if }F(n)=0)
\end{cases}\quad \text{for }\Mc{S}\in\Mf{P}(\Mb{C})
\end{align*}\\
and $X=\Cap_{n\ge1}\Phi_1\circ\Phi_2\circ\cdots\circ\Phi_n(I).$ It is easy to see that $(X,d)$ is a complete, doubling, uniformly perfect metric space, where $d$ is the Euclidean metric on $X$ given by $d(z,w)=|z-w|.$ We also let
\begin{align*}
	T_n=&\begin{cases}
		\{\phi\} & (n=0)\\
		\begin{aligned}
			\{(w_i)_{i=1}^n\mid w_i\in &\{0,1,3,5,7\}\text{ if }F(i)=1,\\
				&w_i\in\{1,...,8\}\text{ if }F(i)=0\}
		\end{aligned} & (\text{otherwise})
	\end{cases}
\end{align*}
and $T=\sCup_{n\ge0} T_n.$ For any $w\in T,$ we define
\begin{align*}
	\varphi_w=& \begin{cases}
		\id_\Mb{C} & (\text{if }w=\phi )\\
		\varphi_{w_1}\circ\cdots\circ\varphi_{w_n} &(\text{otherwise})
	\end{cases}, \hspace{20pt}K_w=\varphi_w(Q)\cap X\\
\text{and } \pi(w)=&\begin{cases}
	\phi &(\text{if }w\in T_0\scup T_1)\\
	(w_i)_{i=1}^{n-1} & (\text{if }w=(w_i)_{i=1}^n\text{ for some }n>1).
\end{cases}
\end{align*}
Then it is easily seen that $K$ is a partition of $(X,d)$ parametrized by $(T,\pi,\phi),$ satisfying the basic framework. Moreover, in the same way as \cite[Proposition 3.11]{Sas2}, we have $\ard(X,d)=\ard(\Mr{SC},d),$ where $\Mr{SC}$ is the standard Sierpi\'nski carpet (recall Figure \ref{FigSC}). It is also routine work to show that there exists a Borel measure $\mu$ such that 
\[\mu(K_w)=3^{-n}(5/3)^{-\#\{k\le n\mid F(k)=1\}}\quad\text{ for every }n\in\Mb{N}\text{ and }w\in T_n,\]
and so $\mu$ is $\vd_d.$ \par
By the fact $2\log5/(\log3+\log5)<1.5<1+(\log2/\log3)\le \ard(\Mr{SC},d)<2$ (see \cite{Tys,TW} for the proof of the last two inequalities) and the above argument, for the proof of Theorem \ref{cex} it suffices to prove the following theorem.
\begin{thm}\label{cres}
	There exists a resistance form $\Mc{(E,F)}$ on $X$ such that the associated resistance metric $R$ is quasisymmetric to the Euclidean metric $d,$ the limit $d_s(\mu,\Mc{E}_\mu,\Mc{D}_\mu)$
	exists and is independent of $x,$ and $d_s(\mu,\Mc{E}_\mu,\Mc{D}_\mu)=2\log5/(\log3+\log5).$
\end{thm}
\begin{rem}
	\begin{enumerate}
		\item $\ard(X,d)=\ard(X,R)$ because $R\qs d.$
		\item $2\log5/(\log3+\log5)$ equals the spectral dimension of the standard Dirichlet form on the Vicsek set (recall that it is the unique nonempty compact subset $\Mr{VS}$ of $\Mb{C}$ with $\Mr{VS}=\cup_{j=0,1,3,5,7}\varphi_j(\Mr{VS})$).
	\end{enumerate}
\end{rem}
Let 
\begin{align*}
	V_0&=\{p_1,p_3,p_5,p_7\},\quad G_0=\{(x,y)\in V_0\times V_0\mid |x-y|=1\},\\
	V_{n,m}&=\begin{cases}
		\Phi_{m+1}\circ\cdots\circ\Phi_{n}(V_0) & (\text{if }m<n)\\
		V_0 & (\text{if }m=n)
	\end{cases}\\
G_{n,m}&=\{(x,y)\in V_{n,m}\times V_{n,m}\mid \text{ for some }w\in T_n,\text{ there exists }(x',y')\in G_0\\
&\hspace{2.5cm}\text{such that }\varphi_w(x')=\varphi_{\pi^{n-m}(w)}(x),\ \varphi_w(y')=\varphi_{\pi^{n-m}(w)}(y)\}
\end{align*}	
for any $n\ge0$ and $0\le m\le n.$ We also define $\Mc{E}_{n,m}:\ell(V_{n,m})\times\ell(V_{n,m})\to\Mb{R}$ by
\[\Mc{E}_{n,m}(u,v)=\frac{1}{2}\sum_{(x,y)\in G_{n,m}}(u(x)-u(y))(v(x)-v(y)),\]
then it is clear that $(\Mc{E}_{n,m},\ell(V_{n,m}))$ are resistance forms. Here $R_{n,m}$ denotes the associated resistance metric and $\Mc{R}_{n,m}$ denotes the resistance between sets. For simplicity of notation, we write 
\begin{align*}
	(\Mr{TB})_{n,m}&=\Mc{R}_{n,m}(\{z\in V_{n,m}\mid \mim(z)=\frac{1}{2}\},\{z\in V_{n,m}\mid \mim(z)=\frac{-1}{2}\})\\
	\text{and}\ (\Mr{Pt})_{n,m}&=R_{n,m}(p_1,p_5).
\end{align*}
Moreover, we write $n$ instead of $n,0$ if no confusion may occur. For example, we write $V_n$ instead of $V_{n,0}.$ In the same way as \cite[Section 4]{BB90} and \cite[Theorem 3.2]{Sas2}, we have the following inequalities.
\begin{lem}\label{resnk}
 There exists $C>0,$ satisfying the following conditions for any $n\ge m\ge0,\ x,y\in V_m$ and $w\in T_n.$
 \begin{enumerate}
 	\item $R_n(x,y)\le CR_m(x,y)(\Mr{Pt})_{n,m}$ and $CR_n(x,y)\ge R_m(x,y)(\Mr{TB})_{n,m}.$
 	\item $C(\Mr{TB})_{n,m}\ge(\Mr{Pt})_{n,m}\ge(\Mr{TB})_{n,m}.$
 	\item Let  $A=\{z\in\Mb{C}\mid |\mre(z)|\vee|\mim(z)|\ge3/2\},$ then \\ 
 	$C\Mc{R}_n(\varphi_w(Q)\cap V_n,\varphi_w(A)\cap V_n)\ge\Mc{R}_m(\varphi_w(Q)\cap V_m,\varphi_w(A)\cap V_m)(\Mr{TB})_{n,m}$
 \end{enumerate} 
In particular, 
\begin{equation}\label{eqpt}
 (\Mr{Pt})_{n}\le C(\Mr{Pt})_{m}(\Mr{Pt})_{n,m}\le C^2(\Mr{Pt})_{m}(\Mr{TB})_{n,m}\le C^3(\Mr{Pt})_{n}
\end{equation}
for any $n\ge m\ge0,$ which follows from (1) and (2).
\end{lem}
\begin{rem}
	$C$ only depends on the structure of the standard 3-adic squares and resistance estimates for the (graphical) standard Sierpi\'nski carpet, so does not depend on $m$ and $n.$
\end{rem}
For the construction of the desired resistance form, we use the following proposition which implicitly appeared and is proved in \cite[Proof of Theorem 5.1]{Kas}.

\begin{prop}\label{finlim}
	Let $V$ be a finite set and $(\Mc{E}_n,\ell(V))$ be a resistance form on $V$ with the associated resistance metric $R_n$ for every $n\ge0.$ If $R(x,y):=\lim_{n\to\infty}R_n(x,y)>0$ exists for any $x,y\in V$ with $x\ne y,$ then there exists a resistance form $(\Mc{E},\ell(V))$ such that the associated resistance metric coincides with $R.$ 
\end{prop}
\begin{cor}\label{lim}
	Let $\{V_n\}$ be a sequence of increasing nonempty finite sets and $(\Mc{E}_n,\ell(V_n))$ be a resistance form on $V_n$ with the associated resistance metric $R_n$ for every $n\ge0.$ If $R(x,y):=\lim_{n\to\infty}R_n(x,y)>0$ exists for any $x,y\in\cup_{n\ge 0}V_n$ with $x\ne y,$ there exists a resistance form on $V_*$ such that the associated resistance metric coincides with $R_*,$ where $(V_*,R_*)$ is the completion of $(\cup_{n\ge0}V_n,R).$
\end{cor}
\begin{proof}
	Applying Proposition \ref{finlim} to $\{(\Mc{E}_m|_{V_n},\ell(V_n))\}_{m\ge n},$ we obtain the resistance form $(\Mc{L}_n,\ell(V_n))$ such that the associated resistance metric coincides with $R|_{V_n\times V_n}.$ By \cite[Theorem 2.1.12]{Kig01} and \cite[Theorem 3.13]{Kig12}, the claim follows.
\end{proof}
Since \eqref{eqpt} and Lemma \ref{resnk} (1) and (2) holds, we have
\begin{equation}\label{inequni}
	C^{-2}\frac{R_m(x,y)}{(\Mr{Pt})_m}\le \frac{R_n(x,y)}{(\Mr{Pt})_n}\le C^3 \frac{R_m(x,y)}{(\Mr{Pt})_m}
\end{equation}
for any $0\le m\le n$ and $x,y\in V_m,$ so by Corollary \ref{lim} with the diagonal sequence argument, we obtain a resistance form $(\Mc{E,F})$ on $V_*,$ such that for some $\{n_j\}_{j\in\Mb{N}}$ the associated resistance metric $R_*$ satisfies
\[R_*(x,y)=\lim_{j\to\infty}\frac{R_{n_j}(x,y)}{(\Mr{Pt})_{n_j}}\]
for any $x,y\in\cup_{n\ge 0}V_n.$\par
In order to show $V_*=X$ and $R\qs d,$ and to calculate $d_s(\mu,\Mc{E}_\mu,\Mc{D}_\mu),$ we need more detailed evaluation as in \cite{Sas2}.
\begin{lem}\label{evres}
	\begin{enumerate}
		\item There exist $M\ge0$ such that $(\Mr{Pt})_{n+M}\ge2(\Mr{Pt})_n\gtrsim (\Mr{Pt})_{n+1}$ for any $n\ge 0.$
		\item For each $x,y\in\cup_{n\ge 0}V_n,$ let
		\[\Delta(x,y)=\min\{n\mid x\in\varphi_w(I) \text{ and }y\in \varphi_w(A)\text{ for some }w\in T_n\},\]
		then $R_*(x,y)\asymp ((\Mr{Pt})_{\Delta(x,y)})^{-1}$ for any $x,y\in\cup_{n\ge 0}V_n.$
	\end{enumerate}
\end{lem}
\begin{proof}
	\begin{enumerate}
		\item Let $k_1(n,m)$ denote $\#\{j\mid m<j\le n, f_*(j)=1\}$ and  $k_2(n,m)$ denote $\#\{j\mid m<j<n, f_*(j)=1, f_*(j+1)=0\}.$ Then in the same way as \cite[Theorem 3.2 (1)]{Sas2} but induction of $(n-k)$ for any fixed $n,$ it follows that there exist $C_a,C_b>0$ and $\rho>1$ such that\begin{equation}\label{eqevpt}
			\rho^{k_1(n,m)}3^{n-m-k_1(n,m)}C_a^{k_2(n,m)}\lesssim (\Mr{Pt})_{n,m}\lesssim \rho^{k_1(n,m)}3^{n-m-k_1(n,m)}C_b^{k_2(n,m)}
		\end{equation}
		
		for any $0\le m\le n,$ because constants do not depend on $n$ and $m.$ Therefore the lemma follows from \eqref{eqpt} and \eqref{eqevpt}.
		\item We first note that 
		\begin{align}
			&R_n(x,y)\asymp 1\text{ for any }n\in\Mb{N}\text{ and }x,y\in V_n\text{ with }(x,y)\in G_n, \label{pt1}\\
			&\Mc{R}_n(\varphi_w(I),\varphi_w(A))\asymp1\text{ for any }n\in\Mb{N}\text{ and }w\in T_n. \label{anu1}
		\end{align}
	By definition of $\Delta(x,y),$ for any $n$ with $x,y\in V_n,$ we have $\{x_i\}_{i=\Delta(x,y)-2}^n$ and $\{y_i\}_{i=\Delta(x,y)-2}^n$ such that $x_{\Delta(x,y)-2}=y_{\Delta(x,y)-2},$ $x_n=x,$ $y_n=y$ and for any $\Delta(x,y)-i\le i\le n,$ $(x_{i-1},x_i)\in \varphi_{w_i}(V_0)$ and $(y_{i-1},y_i)\in \varphi_{u_i}(V_0)$ for some $w_i,u_i\in T_i.$ Then by (1) and \eqref{inequni} and \eqref{pt1},
	\begin{align*}
		\frac{R_n(x,y)}{(\Mr{Pt})_n}&\le \frac{1}{(\Mr{Pt})_n}\sum_{i=\Delta(x,y)-1}^n (R_n(x_{i-1},x_i)+R_n(y_{i-1},y_i))\\
		&\lesssim \sum_{i=\Delta(x,y)-1}^n \frac{1}{(\Mr{Pt})_i}\\
		&\lesssim \sum_{i=0}^\infty\frac{\sum_{j=1}^M(\Mr{Pt})_{\Delta(x,y)-1+iM+j,\Delta(x,y)-1+iM}}{(\Mr{Pt})_{\Delta(x,y)-1+iM}}\\
		&\lesssim ((\Mr{Pt})_{\Delta(x,y)-1})^{-1}\ \lesssim((\Mr{Pt})_{\Delta(x,y)})^{-1}
	\end{align*}
	for any $n\in\Mb{N}$ and $x,y\in V_n.$\\
	On the other hand, let $w\in T_{\Delta(x,y)}$ be a vertex appeared in the definition of $\Delta(x,y),$ then by \eqref{anu1} and Lemma \ref{resnk} (3),
	\begin{align*}
		\frac{R_n(x,y)}{(\Mr{Pt})_n}&\ge \frac{\Mc{R}_n(\varphi_w(I), \varphi_w(A) )
		}{(\Mr{Pt})_n}\\
	&\gtrsim \frac{(\Mr{Pt})_{n,\Delta(x,y)}}{(\Mr{Pt})_n}\ \gtrsim ((\Mr{Pt})_{\Delta(x,y)})^{-1},
	\end{align*}
	which completes the proof.
	\end{enumerate}
\end{proof}
\begin{proof}[Proof of Theorem \ref{cres}]
	We first prove $d|_{\cup_{n\ge0}V_n}\qs R_*|_{\cup_{n\ge0}V_n}.$ By Lemma \ref{evres} (1) and (2), there exist $\alpha,\tau>1$ such that 
	\begin{align*}
		\bullet & \text{ if }\Delta(x,y)-\Delta (x,z)\ge Ma\text{ for } a\ge0 \text{ then }\frac{R_*(x,y)}{R_*(x,z)}\le \alpha2^{-a}.\\
		\bullet & \text{ If }\Delta(x,z)-\Delta (x,y)\le a\text{ for } a\ge0 \text{ then }\frac{R_*(x,y)}{R_*(x,z)}\le \alpha\tau^{a}.
	\end{align*}
	Since $\Delta(x,z)-\Delta(x,y)\le \log(6\sqrt{2}d(x,y)/d(x,z))/\log 3$ and the above inequalities hold, there exist $t_1, t_2$ with $0<t_1<t_2$ such that 
	\begin{align*}
		\bullet &\text{ if }d(x,y)/d(x,z)\le t_1\text{ then } R_*(x,y)/R_*(x,z)\le\theta_1(d(x,y)/d(x,z)),\\
		\bullet &\ R_*(x,y)/R_*(x,z)\le\theta_2\bigl((d(x,y)/d(x,z))\vee t_2\bigr)
	\end{align*}
	for $x,y,z$ with $x\ne z,$ where
	\[\theta_1(t)=\alpha2^{(\log 6\sqrt{2}t/M\log3)+1},\quad \theta_2(t)=\alpha\tau^{(\log 6\sqrt{2}t/\log3)+1}.\]
	It is obvious that there exists a homeomorphism $\theta:[0,\infty)\to[0,\infty)$ satisfying $\theta_1(t)\le\theta(t)$ for $t\le t_1$ and  $\theta_2(t\vee t_2)\le\theta(t)$ for $t>t_1,$ which proves the desired quasisymmetry. This also shows a sequence in $\cup_{n\ge0} V_n$ is $d$-Cauchy if and only if $R_*$-Cauchy, therefore $V_*=X$ and $d\qs R_*.$ In other words, $(\Mc{E,F})$ is the desired resistance form.\par
	It remains to calculate $d_s(\mu,\Mc{E}_\mu,\Mc{D}_\mu)$. By Proposition \ref{pre} we can apply \cite[Theorem 15.6]{Kig12} for $d$ and obtain
	\begin{align*}
		&\limsup_{t\to\infty}\frac{-\log p_\mu(1/t,x,x)}{\log t}\\
		=&\limsup_{n\to\infty}\frac{\log V_d(x,3^{-n})}{\log h_d(x,3^{-n})}\\
		\le&\limsup_{n\to\infty}\frac{\frac{k_1(n)}{n}\log8+(1-\frac{k_1(n)}{n})\log5}{\frac{k_1(n)}{n}(\log8+\log\rho)+(1-\frac{k_1(n)}{n})(\log5+\log3)+\frac{k_2(n)}{n}\log C_a}\\
		=&\frac{\log5}{\log5+\log3}
	\end{align*}	
for any $x\in X.$ (Note that the first equation follows from Lemma \ref{exp}.) Similarly we have
\[\liminf_{t\to\infty}\frac{-\log p_\mu(1/t,x,x)}{\log t}\ge\frac{\log5}{\log5+\log3},\]
and the proof is complete.
\end{proof}

\appendix
\section{Equivalence of local properties}\label{ApLoc}
Let us recall that $X$ is a set, $(\Mc{E,F})$ is a resistance form on $X$ and $R$ is the resistance metric associated with $(\Mc{E,F}).$ In this appendix we discuss the relation between the local property of $(\Mc{E,F})$ and that of the Dirichlet form induced by $(\Mc{E,F}).$ We also recall that the local property of a Dirichlet form is defined as follows.
\begin{defi}[Local]
	Let $(Y,\rho)$ be a locally compact separable metric space and $\nu$ be a Radon measure on $Y$ with full support. A Dirichlet form $(E,D)$ on $L^2(Y,\nu)$ is called \emph{local} if $E(u,v)=0$ whenever $u,v\in D$ have disjoint compact supports, where the \emph{support} $\Mr{supp}(u)$ of $u\in L^2(Y,\nu)$ is defined as the support of the measure $ud\nu$ on $(Y,\rho).$
\end{defi}

By \cite[Theorem 9.4]{Kig12}, if $(\Mc{E,F})$ is a regular resistance form satisfying (ACC), then for each Radon measure $\mu$ on $X$ with full support, $(\Mc{E}_\mu,\Mc{D}_\mu),$ defined in the same way as Lemma \ref{xhk}, is a regular Dirichlet form on $L^2(X,\mu).$ Here we remark that $\Mr{supp}(u)=\ol{\{x\in X\mid u(x)\ne 0\}}$ because $\Mc{F}\subset C(X,R).$ Therefore by the definition of $\Mc{D}_\mu$, $(\Mc{E}_\mu,\Mc{D}_\mu)$ is a local Dirichlet form (over $(X,R)$) if $(\Mc{E,F})$ is a local resistance form. In this appendix we prove that, under Assumption \ref{ass}, the converse direction is also true. Indeed, the following holds.
\begin{prop}\label{PropDen}
	Assume that $R$ is complete and doubling. Then for any $u\in \Mc{F},$ there exists $\{u_n\}_{n\ge0}\subset\Mc{F}\cap C_0(X,R)$ such that $\Mr{supp}(u_n)\subset \Mr{supp}(u)$ for any $n$ and $\lim_{n\to\infty}\Mc{E}(u-u_n,u-u_n)=0.$
\end{prop}
\begin{cor}\label{CorLoc}
We make the Assumption \ref{ass} and let $\mu$ be a Radon measure on $X$ with full support. Then the following conditions are equivalent.
\begin{enumerate}
	\item $\Mc{E}(u,v)=0$ if $u,v\in\Mc{F}$ and $\Mr{supp}(u)\cap\Mr{supp}(v)=\emptyset.$
	\item $(\Mc{E,F})$ is a local resistance form.
	\item $(\Mc{E}_\mu,\Mc{D}_\mu)$ is a local Dirichlet form.
\end{enumerate} 
\end{cor}
\begin{proof}
	$(1)\Rightarrow(2)\Rightarrow(3)$ is obvious. $(3)\Rightarrow(1)$ follows from Theorem \ref{xbf} and Proposition \ref{PropDen}.
\end{proof}
In the remainder of this appendix, we assume that $R$ is complete and doubling, and prove Proposition \ref{PropDen}. In the same way as the proof of Lemma \ref{resbb}, the following inequality holds without the uniform perfectness condition.
\begin{lem}\label{LemReslow}
	$\Mc{R}(\ol{B(x,r)}, B(x,2r)^c)\gtrsim r$ for any $x\in X$ and  $r>0.$
\end{lem} 
The following Proposition \ref{PropKaj1}, Corollary \ref{CorKaj2} and Lemma \ref{LemKaj3} were proved in \cite{Kaj} for a general resistance form whose associated resistance metric is not necessarily doubling. Here we give proofs for the same reason as we did for Lemma \ref{inf}.
\begin{prop}[{cf.~\cite[Theorem 2.38 (2)]{Kaj}}]\label{PropKaj1}
Let $u\in\Mc{F}$ and $\{u_n\}_{n\ge0}\subset \Mc{F}.$ Then $\lim_{n\to\infty}\Mc{E}(u-u_n,u-u_n)=0$ if and only if $\limsup_{n\to\infty}\Mc{E}(u_n,u_n)\le\Mc{E}(u,u)$ and $\lim_{n\to\infty}(u-u_n)(x)$ exists in $\Mb{R}$ and is constant on $X.$
\end{prop}
\begin{proof}
	The necessity is clear by the triangle inequality of $\Mc{E}^{1/2}$ and \eqref{RF3}. For the sufficiency, let $\{V_m\}_{m\ge0}$ be a spread sequence of $(X,R)$ then
	\begin{multline*}
		\Mc{E}(u,u)=\lim_{m\to\infty}\Mc{E}|_{V_m}(u|_{V_m},u|_{V_m})\\=\lim_{m\to\infty}\lim_{n\to\infty}\Mc{E}|_{V_m}(u_n|_{V_m},u_n|_{V_m})\le\liminf_{n\to\infty}\Mc{E}(u_n,u_n),
	\end{multline*}
	which proves $\lim_{n\to\infty}\Mc{E}(u_n,u_n)=\Mc{E}(u,u).$ Let $u^*_n:=(u+u_n)/2,$ then by the triangle inequality of $\Mc{E}^{1/2},$ $\{u^*_n\}_{n\ge0}$ also satisfies the same condition as $\{u_n\}_{n\ge0}.$ Therefore $\lim_{n\to\infty}\Mc{E}(u^*_n,u^*_n)=\Mc{E}(u,u)$ and
	\begin{align*}
		\lim_{n\to\infty}\Mc{E}(u-u_n,u-u_n)=2\Mc{E}(u,u)+\lim_{n\to\infty}(2\Mc{E}(u_n,u_n)-\Mc{E}(2u^*_n,2u^*_n))=0.
	\end{align*}
\end{proof}
\begin{cor}[{cf.~\cite[Corollary 2.39 (4)]{Kaj}}]\label{CorKaj2}
$\lim_{n\to\infty}\Mc{E}(u-\hat u_n,u-\hat u_n)=0$ for any $u\in\Mc{F},$ where $\hat u_n=(u\wedge n)\vee (-n).$ 
\end{cor}
\begin{proof}
	It immediately follows from Proposition \ref{PropKaj1} and \eqref{RF4}. 
\end{proof}
\begin{lem}[{cf.~\cite[Corollary 2.39 (3)]{Kaj}}]\label{LemKaj3}
	Let $u,v\in\Mc{F}$ be bounded. Then $uv\in\Mc{F}$ and 
	$\Mc{E}(uv,uv)^{1/2}\le \|u\|_{\infty}\Mc{E}(v,v)^{1/2}+\|v\|_{\infty}\Mc{E}(u,u)^{1/2},$ where $\|u\|_{\infty}=\sup_{x\in X}|u(x)|.$
\end{lem}
\begin{proof}
	This follows from Proposition \ref{comp} with easy calculation.
\end{proof}
\begin{proof}[Proof of Proposition \ref{PropDen}]
	Since Corollary \ref{CorKaj2} holds, we only need to show the case that $u\in\Mc{F}$ is bounded and $\diam(X,R)=\infty$. Fix some $x\in X$ and let $f_n$ be the optimal function for $\Mc{R}(\ol{B(x,2^n)}, B(x,2^{n+1})^c)$ for each $n\ge0.$ Let $u_n=f_nu,$ then $u_n\in C_0(X,R)$ and $\Mr{supp}(u_n)\subset \Mr{supp}(u).$ Moreover, $u_n\in\Mc{F}$ and 
	\begin{align*}
		\limsup_{n\to\infty}\Mc{E}(u_n,u_n)&\le \bigl( \Mc{E}(u,u)^{1/2}+\|u\|_{\infty}\lim_{n\to\infty}\Mc{E}(f_n,f_n)^{1/2}\bigr)^2=\Mc{E}(u,u)
	\end{align*}
because of Lemmas \ref{LemReslow} and \ref{LemKaj3}. Therefore $\lim_{n\to\infty}\Mc{E}(u-u_n,u-u_n)=0$ by Proposition \ref{PropKaj1}, which completes the proof.
\end{proof}
\section*{Acknowledgments}
I would like to thank Naotaka Kajino for his kind grammatical corrections to the introduction and technical comments. In particular, Appendix A was motivated by his remark. I also thank my supervisor Takashi Kumagai for his comments on the introduction.
This work was supported by JSPS KAKENHI Grant Number JP20J23120. 

\end{document}